\documentclass[reqno, 11pt]{amsart}
\usepackage{amssymb}
\usepackage{enumerate}

\numberwithin{equation}{section}

\newtheorem{theorem}{Theorem}[section]
\newtheorem{corollary}[theorem]{Corollary}
\newtheorem{lemma}[theorem]{Lemma}
\newtheorem{proposition}[theorem]{Proposition}

\theoremstyle{definition}
\newtheorem{remark}[theorem]{Remark}

\theoremstyle{definition}

\theoremstyle{definition}
\newtheorem{assumption}[theorem]{Assumption}

\makeatletter
\newcommand\dashint{\operatorname%
{\,\,\text{\bf--}\kern-.98em\DOTSI\intop\ilimits@\!\!}}
\makeatother

\newcommand{\osc}{\mathop{\hbox{osc}}}

\newcommand\vu{\textit{\textbf{u}}}

\newcommand\vf{\textit{\textbf{f}}}
\newcommand\vg{\textit{\textbf{g}}}

\newcommand\bR{\mathbb{R}}

\newcommand\bH{\mathbb{H}}

\newcommand\bM{\mathbb{M}}

\newcommand\cD{\mathcal{D}}

\newcommand\cH{\mathcal{H}}
\newcommand\cP{\mathcal{P}}

\newcommand\cQ{\mathcal{Q}}

\newcommand\cL{\mathcal{L}}

\newcommand{\set}[1]{\left\{#1\right\}}

\newcommand{\Div}{\operatorname{div}}

\begin{document}

\title[Solvability of second-order equations]{Solvability of second-order equations with hierarchically partially BMO coefficients}

\author[H. Dong]{Hongjie Dong}
\address[H. Dong]{Division of Applied Mathematics, Brown University,
182 George Street, Providence, RI 02912, USA}
\email{Hongjie\_Dong@brown.edu}
\thanks{H. Dong was partially supported by  NSF grant number DMS-0635607 from IAS, and NSF grant number DMS-0800129.}



\begin{abstract}
By using some recent results for divergence form equations obtained in \cite{Dong08,DongKim08b},
we study the $L_p$-solvability of second-order elliptic and parabolic equations in nondivergence form for any $p\in (1,\infty)$. The leading coefficients are assumed to be in locally BMO spaces with suitably small BMO seminorms. We not only extend several previous results by Krylov and Kim \cite{Kim07}-\cite{KK2} to the full range of $p$, but also deal with equations with more general coefficients.
\end{abstract}

\maketitle

\section{Introduction}

In this paper, we study the $W^{1,2}_{p}$-solvability of parabolic equations in nondivergence form:
\begin{equation}
                                                \label{parabolic}
P u-\lambda u=f,
\end{equation}
where $\lambda\ge 0$ is a constant, $f\in L_p$, and
\begin{equation*}
P u=-u_t+a^{jk}D_{jk}u+b^{j}D_ju+cu.
\end{equation*}
We assume that
all the coefficients are bounded and measurable, and $a^{jk}$ are uniformly elliptic, i.e., for some $K>0$ and $\delta\in (0,1]$,
$$
|b^j|+|c|\le  K,\quad |a^{jk}|\le \delta,
\quad\delta|\xi|^{2}\leq
a^{jk}\xi^{j}\xi^{k}\leq\delta^{-1}|\xi|^{2}.
$$
If all the coefficients are time-independent, we also consider
the
$W^{2}_{p}$-solvability of elliptic equations in nondivergence
form:
\begin{equation}
                                                \label{elliptic}
L u-\lambda u=f,
\end{equation}
where
$$ L u=a^{jk}D_{jk}u+b^{j}D_ju+cu.
$$
The main purpose of this paper is to show how the recent
results obtained in \cite{Dong08,DongKim08b} for {\em divergence form} equations can
be used to i) overcome the
restriction $p>2$ in the $L_p$-theory of {\em nondivergence form}
equations with partially BMO coefficients
developed by Krylov and Kim in \cite{Kim07}-\cite{KK2}; ii) extend the results to systems as well as equations with a more general class of coefficients.

There is a vast literature of the $L_p$ theory of second-order parabolic and elliptic equations with discontinuous coefficients. It is of particular interest because of its various important applications in nonlinear equations and its subtle links with the theory of stochastic processes. For instance, one implication of the $L_p$-solvability of parabolic equations in nondivergence form is the weak uniqueness of solutions to associated It\^o equations.

For equations with uniformly continuous leading coefficients, the solvability is classical. The $L_p$ theory of second-order
equations with discontinuous coefficients was studied
extensively in the last three decades. One important class of
discontinuous coefficients contains functions with vanishing
mean oscillation (VMO), the study of which was initiated in
\cite{CFL1} about twenty years ago and continued in \cite{CFL2} and \cite{BC93};
see also \cite{MaPaSo00,PS2,HHH} and the references in \cite{Krylov_2007_mixed_VMO}.
The proofs in \cite{CFL1,CFL2,BC93} are based on the Calder\'on--Zygmund estimate and the Coifman--Rochberg--Weiss commutator estimate. Before that, the $L_p$ theory had been established for some other types of discontinuous coefficients; see, for
instance, \cite{Lo72a} and \cite{Chi}.

In \cite{Krylov_2005}, Krylov gave a unified
approach to investigating the $L_p$-solvability of both
divergence and nondivergence form parabolic and elliptic equations with $a^{ij}$ VMO in the spatial
variables (and measurable in the time variable in the parabolic case).
The proofs in
\cite{Krylov_2005} rely mainly on pointwise estimates of sharp
functions of spatial derivatives of solutions together with the Stein--Fefferman theorem and the Hardy--Littlewood theorem. By doing this, VMO coefficients are treated in a rather straightforward
manner. This result was later improved and generalized in a
series of papers \cite{Dong08}-\cite{DongKrylov}, \cite{Kim07}-\cite{KK2}, and \cite{Krylov_2007_mixed_VMO}-\cite{Krylov08}\footnote{Although the results in some of these papers are claimed for equations with VMO coefficients, the proofs there only require $a^{ij}$ to have locally small mean oscillations.}. For other results about equations and systems with VMO/BMO coefficients, we refer the reader to \cite{AM07,ByunWang08} and references therein.

In contrast, the theory of elliptic and parabolic equations with partially
VMO coefficients is quite new, and was originated in \cite{KimKrylov07}. In \cite{KimKrylov07} the $W^2_p$-solvability for any $p> 2$ was established for
nondivergence form elliptic equations with leading coefficients
measurable in one variable and VMO in the others. This result
was extended in \cite{KK2} to parabolic equations. For nondivergence form parabolic equations, further extensions were made later in
\cite{Kim07,Kim07a,Kim07b}, in which most leading coefficients are
measurable in the time variable as well as one spatial
variable, and VMO, or BMO with small seminorms, in the other variables. In all these papers, except for some special cases, it is always assumed that $p>2$. The corresponding results for divergence form equations were proved more recently in \cite{Dong08,DongKim08a,DongKim08b} by using the idea of ``breaking the symmetry'' of the coordinates and applying a generalized Stein--Fefferman theorem proved in \cite{Krylov08}. In these papers, the solvability is obtained for any $p\in (1,\infty)$ thanks to a standard duality argument for divergence form equations.

Roughly speaking, the restriction $p>2$ mentioned above is due to the following reason. In \cite{KimKrylov07}, for instance, a sharp function estimate is deduced from the $W^2_2$-solvability of equations with $a^{ij}$ depending only on $x^1$, which is obtained by using the method of Fourier transforms. In turn, the right-hand side of the estimate contains maximal functions of $q$-th power of $D^2 u$ for some $q>2$, which can be made arbitrarily close to $2$. Therefore, to apply the Fefferman--Stein theorem and the Hardy--Littlewood maximal function theorem one requires $p\ge q>2$.

Thus, a natural question is:
\begin{quote}
{\em Do we still have the $L_p$-solvability for nondivergence equations with $a^{ij}$
in the class of \cite{Kim07}-\cite{KimKrylov07} when $p$ gets below $2$?}
\end{quote}
In other word, we would like to understand whether the condition $p>2$ in these papers is necessary, or merely due to technical reasons in the method used. This is one of the motivations of this article.
A partial answer to this question can be found in Kunstmann \cite{kunst}, in which the author extended the result in \cite{KK2} from $p\in (2,\infty)$ to $p\in (p_0,\infty)$ for some $p_0\in (1,2)$ in the special case that $a^{ij}$ are time-independent.

In this article, we shall give an affirmative answer to this question in full.
One of our main results is the $W^{1,2}_p$-solvability of nondivergence form parabolic equations for any $p\in (1,\infty)$ (stated in Theorems \ref{thm2} and \ref{thm3}), under the assumptions that $a^{ij}$ are partially BMO. More precisely, $a^{ij}$ are assumed to be measurable in
one spatial direction and the time variable, and have locally small mean oscillations in the other variables; additionally, we assume $a^{11}$ to be measurable in one spatial direction (or the time variable) and have locally small mean oscillations in the others; see Section \ref{mainsec} for the definitions. Consequently, we derive the $W^2_p$-solvability of nondivergence elliptic equations for any $p\in (1,\infty)$, when $a^{ij}$ are in the class of \cite{KimKrylov07}. We also obtain the solvability for nondivergence form parabolic and elliptic {\em systems} with partially BMO coefficients, which, to the best of our knowledge, is new even when $p>2$; see Remark \ref{rem4.32}. Moreover, the method used in this paper allows us to treat more general coefficients, i.e. ``hierarchically'' partially BMO coefficients; see Section \ref{sec6}.

Now we give a brief description of our methods.
In order to get below $2$, we first establish for any $p\in (1,\infty)$ the $W^{1,2}_p$-solvability of parabolic equations with $a^{11}$ depending only on either $t$ or $x^1$, and $a^{ij},ij>1$ depending only on $(t,x^1)$ as stated in Theorems \ref{thm2.2} and \ref{thm3.2}. To this end, our idea is that in this situation the equation can be rewritten into a divergence form after a suitable change of variables, so that some results recently proved in \cite{Dong08} and \cite{DongKim08b} for divergence form equations can be applied; see Appendix. In some sense, this idea is reminiscent of the classical argument of deriving from the De Giorgi--Moser--Nash estimate the $C^{1,\alpha}$ regularity for nondivergence elliptic equations with {\em measurable} coefficients on the plane; see, for instance, \cite[\S 11.2]{GT98}.
Secondly, we bound $D^2 u$ and $u_t$ by a portion of $D^2u$ and $Pu$.
Finally, we get a sharp function estimate of the portion of $D^2u$ by using the Krylov--Safonov interior H\"older estimate. The above three steps are combined together to prove Theorems \ref{thm2} and \ref{thm3} by applying the Fefferman--Stein theorem and the Hardy--Littlewood maximal function theorem.
Of course, the Krylov--Safonov estimate is not applicable to systems, for which case we use a bootstrap argument to get an interior H\"older estimate.
To deal with equations with hierarchically partially BMO coefficients, in Section \ref{sec6} we first need to establish the corresponding solvability results for divergence form equations with the leading coefficients in the same class. This is achieved by exploiting the aforementioned idea in \cite{Dong08,DongKim08a,DongKim08b} of ``breaking the symmetry'' of the coordinates and running an induction argument on the number of the coordinates.

We note that the $W^{1,2}_p$-solvability results in this paper admit an extension to the mixed norm spaces $W^{1,2}_{q,p}$ by following the idea in \cite{Krylov_2007_mixed_VMO}; see also, for instance,  \cite{Kim07b} and \cite{Dong08}. Here we do not pursue this, and leave it to the interested reader.

A brief outline of the paper: in the next section, we introduce the notation and state some main results, Theorems \ref{thm2}, \ref{thm3} and \ref{thm4}. Section \ref{sec3} and \ref{sec4} are devoted to the proof of Theorems \ref{thm2} and \ref{thm3}. We give a generalization of Theorems \ref{thm2} to systems in Section \ref{sec5}. Finally, in Section \ref{sec6} we discuss equations with hierarchically partially BMO coefficients.


\section{Notation and some main results}		\label{mainsec}

\subsection{Notation}
Let $d\geq 1$ be an integer. A typical point in $\bR^{d}$ is denoted by $x=(x^1,\ldots,x^d)=(x^1,x')$.
We set
$$
D_{i}u=u_{x^i},\quad D_{ij}u=u_{x^ix^j},\quad
D_t u=u_t.
$$
By $Du$ and $D^{2}u$ we
mean the gradient and the Hessian matrix
of $u$. On many occasions we need to take these objects relative to only part of variables. We also use the following notation:
$$
D_{x'}u=u_{x'},\quad
D_{x^1x'}u=u_{x^1x'},\quad
D_{xx'}u=u_{xx'}.
$$
Throughout the paper, we always assume that $1 < p, q < \infty$
unless explicitly specified otherwise.
By $N(d,p,\ldots)$ we mean that $N$ is a constant depending only
on the prescribed quantities $d, p,l,\ldots$.
For a function $f(t,x)$ in $\bR^{d+1}$, we set
\begin{equation*}
(f)_{\cD} = \frac{1}{|\cD|} \int_{\cD} f(t,x) \, dx \, dt
= \dashint_{\cD} f(t,x) \, dx \, dt,
\end{equation*}
where $\cD$ is an open subset in $\bR^{d+1}$ and $|\cD|$ is the
$d+1$-dimensional Lebesgue measure of $\cD$.
For $-\infty\leq S<T\leq \infty$, we denote
\begin{align*}
W_{p}^{1,2}((S,T)\times \bR^d)&=
\set{u:\,u,u_t,Du,D^2u\in L_{p}((S,T)\times \bR^d)},\\
\cH^{1}_{p}((S,T)\times \bR^d)&=(1-\Delta)^{1/2}W_{p}^{1,2}((S,T)\times \bR^d),\\
\bH^{-1}_{p}((S,T)\times \bR^d)&=(1-\Delta)^{1/2}L_{p}((S,T)\times \bR^d).
\end{align*}
We also use the abbreviations $L_p=L_p(\bR^{d+1})$, $\cH^1_p=\cH^1_p(\bR^{d+1})$, etc. For any $T\in (-\infty,\infty]$, we denote
\begin{equation*}
\bR_T=(-\infty,T), \quad \bR_T^{d+1}=\bR_T\times \bR^d.
\end{equation*}

For any integer $k\ge 1$ and $x\in \bR^k$, we denote $B_r^k(x)$ to be the $k$-dimensional cube
$$
\{y\in \bR^k\,|\,\max_i|y^i-x^i|<r\}.
$$
Set
$$
Q_r^k(t,x) = (t-r^2,t) \times B_r^k(x),\quad
B_r^k=B_r^k(0),\quad
Q_r^k = Q_r^k(0,0).
$$
In case $k=d$ or $d-1$, we use the abbreviations
$$
B_r(x)=B_r^d(x),\quad
Q_r(t,x)=Q_r^d(t,x),
$$
$$
B_r'(x')=B_r^{d-1}(x'),\quad
Q_r'(t,x')=Q_r^{d-1}(t,x').
$$
Let
$$
\cQ=\set{Q_r(t,x): (t,x) \in \bR^{d+1}, r \in (0, \infty)}.
$$

For a function $g$ defined on $\bR^{d+1}$,
we denote its (parabolic) maximal and sharp function, respectively, by
\begin{align*}
\bM g (t,x) &= \sup_{Q\in \cQ: (t,x) \in Q}
\dashint_{Q} | g(s,y) | \, dy \, ds,\\
g^{\#}(t,x) &= \sup_{Q\in \cQ:(t,x) \in Q}
\dashint_{Q} | g(s,y) - (g)_Q | \, dy \, ds.
\end{align*}

\subsection{Main results without the hierarchical structure}
To illustrate the main idea, first we deal with the coefficients considered before in \cite{Kim07}-\cite{KK2}. We assume that $a^{ij},ij>1$ are measurable in
$x^1$ and $t$, and have locally small mean oscillations in the other variables. In addition, we assume that $a^{11}$ are measurable in $t$ (or $x^1$) and have locally small mean oscillations in the others. To state the assumptions on $a^{ij}$ precisely, for $R>0$, we denote
\begin{align*}
a^{11,\#}_{R,1}&=\sup_{(t_0,x_0)\in \bR^{d+1}} \sup_{r\le R}\dashint_{Q_r(t_0,x_0)}|a^{11}(t,x)-\bar a^{11}(t)|\,dx\,dt,\\
a^{11,\#}_{R,2}&=\sup_{(t_0,x_0)\in \bR^{d+1}} \sup_{r\le R}\dashint_{Q_r(t_0,x_0)}|a^{11}(t,x)-\hat a^{11}(x^1)|\,dx\,dt,\\
a^{\#}_R&=
\sup_{(t_0,x_0)\in \bR^{d+1}} \sup_{r\le R}\sup_{(i,j)\neq (1,1)}\dashint_{Q_r(t_0,x_0)}|a^{ij}(t,x)-\bar a^{ij}(t,x^1)|\,dx\,dt,
\end{align*}
where for each $Q_r(t_0,x_0)$,
\begin{align*}
\bar a^{11}(t)&=\dashint_{B_r(x_0)}a^{11}(t,x)\,dx,\\
\hat a^{11}(x^1)&=\dashint_{Q_r'(t_0,x_0')}a^{11}(s,x^1,y')\,dy'\,ds,\\
\bar a^{ij}(t,x^1)&=\dashint_{B_r'(x_0')}a^{ij}(t,x^1,y')\,dy',\quad (i,j)\neq (1,1).
\end{align*}
We shall impose part of the following assumptions on the leading coefficients.
\begin{assumption}[$\gamma$] \label{assump2}$\quad$

\begin{enumerate}[i)]
\item There exists a positive constant $R_0$ such that $a^{11,\#}_{R_0,1}\le \gamma$.
\item There exists a positive constant $R_0$ such that $a^{11,\#}_{R_0,2}\le \gamma$.
\item There exists a positive constant $R_0$ such that
$a^{\#}_{R_0}\le \gamma$.
\end{enumerate}
\end{assumption}

Next we state a few results of the paper.

\begin{theorem}
                                         \label{thm2}

For any
$p\in (1,\infty)$, there exists a $\gamma =\gamma(d,\delta,p) >0$
such that under Assumption
\ref{assump2} ($\gamma$) i) and iii) for any $T\in (-\infty,+\infty]$ the following holds.

i) For any $u\in W^{1,2}_p(\bR^{d+1}_T)$,
$$
\lambda\|u\|_{L_{p}(\bR^{d+1}_T)}+\sqrt{\lambda}
\|Du\|_{L_{p}(\bR^{d+1}_T)}+\|D^{2}u\|_{L_{p}(\bR^{d+1}_T)}
+\|u_t\|_{L_{p}(\bR^{d+1}_T)}
$$
\begin{equation}
                                                \label{13.10.21}
\leq N\|Pu-\lambda u\|_{L_{p}(\bR^{d+1}_T)},
\end{equation}
provided that $\lambda\geq \lambda_0$, where $\lambda_0
\geq0$ depends only
 on $d,\delta,p,K$, and $R_0$, and $N$ depends only on $d,\delta$ and $p$.

ii) For any $\lambda> \lambda_0$ and $f\in L_p(\bR^{d+1}_T)$, there exists a unique solution
$u\in W^{1,2}_p(\bR^{d+1}_T)$ of equation \eqref{parabolic} in
$\bR^{d+1}_T$.

iii) In the case that
$a^{11}=a^{11}(t)$, $a^{jk}=a^{jk}(t,x^1),jk>1$ and $b^j\equiv c\equiv 0$, we can take
$\lambda_0=0$ in i) and ii).
\end{theorem}

In Section \ref{sec5}, we will extend Theorem \ref{thm2} to parabolic systems under the same regularity assumption on the leading coefficients.

\begin{theorem}
                                         \label{thm3}

For any
$p\in (1,\infty)$, there exists a $\gamma =\gamma(d,\delta,p) >0$
such that under Assumption
\ref{assump2} ($\gamma$) ii) and iii) for any $T\in (-\infty,+\infty]$ the following holds.

i) For any $u\in W^{1,2}_p(\bR^{d+1}_T)$, we have \eqref{13.10.21}
provided that $\lambda\geq \lambda_0$, where $\lambda_0
\geq0$ depends only  on $d,\delta,p,K$, and $R_0$, and $N$ depends only on $d,\delta$ and $p$.

ii) For any $\lambda> \lambda_0$ and $f\in L_p(\bR^{d+1}_T)$, there exists a unique solution
$u\in W^{1,2}_p(\bR^{d+1}_T)$ of equation \eqref{parabolic} in
$\bR^{d+1}_T$.

iii) In the case that
$a^{11}=a^{11}(x^1)$, $a^{jk}=a^{jk}(t,x^1),jk>1$ and $b^j\equiv c\equiv 0$, we can take
$\lambda_0=0$ in i) and ii).
\end{theorem}

As a consequence of Theorem \ref{thm3}, we obtain the $W^2_p$-solvability of second-order elliptic equations in nondivergence form with partially BMO coefficients with locally small semi-norms. For this purpose, we assume that $a^{ij}$, $b^i$, $c$ and $f$ are independent of $t$.
For $R>0$ now we denote
$$
a^{\#}_R=
\sup_{x_0\in \bR^{d}} \sup_{r\le R}\sup_{i,j}\dashint_{B_r(x_0)}|a^{ij}(x)-\bar a^{ij}(x^1)|\,dx,
$$
where for each $B_r(x_0)$ and $(i,j)$,
$$
\bar a^{ij}(x^1)=\dashint_{B_r'(x_0')}a^{ij}(x^1,y')\,dy'.
$$
\begin{assumption}[$\gamma$]
                                            \label{assump4}
There exists a positive constant $R_0$ such that $a^{\#}_{R_0}\le \gamma$.
\end{assumption}

\begin{theorem}
                                         \label{thm4}

For any
$p\in (1,\infty)$, there exists a $\gamma =\gamma(d,\delta,p) >0$
such that under Assumption \ref{assump4} ($\gamma$) the following holds.

i) For any $u\in W^{2}_p(\bR^{d})$,
\begin{equation*}
\lambda\|u\|_{L_{p}(\bR^{d})}+\sqrt{\lambda}
\|Du\|_{L_{p}(\bR^{d})}+\|D^{2}u\|_{L_{p}(\bR^{d})}
\leq N\|Lu-\lambda u\|_{L_{p}(\bR^{d})},
\end{equation*}
provided that $\lambda\geq \lambda_0$, where $\lambda_0
\geq0$ depend only
 on $d,\delta,p,K$ and $R_0$, and $N$ depends only on $d,\delta$ and $p$.

ii) For any $\lambda> \lambda_0$ and $f\in L_p(\bR^{d})$, there exists a unique solution
$u\in W^{2}_p(\bR^{d})$ of equation \eqref{elliptic} in
$\bR^{d}$.

iii) In the case that
$a^{jk}=a^{jk}(x^1)$ and $b^j\equiv c\equiv 0$, we can take
$\lambda_0=0$ in i) and ii).
\end{theorem}

In Section \ref{sec6}, we shall extend Theorems \ref{thm2}, \ref{thm3} and \ref{thm4} to equations with more general coefficients. Roughly speaking, we allow hierarchically partially BMO coefficients. This is also one of the main advances of this paper. For example, when $d=3$, Assumption \ref{assump7b} means: $a^{11}$ is measurable in $x^1$ and has small mean oscillation in $(t,x^2,x^3)$; $a^{12}$, $a^{21}$ and $a^{22}$ are measurable in $(t,x^1)$ and has small mean oscillation in $(x^2,x^3)$; finally, all the other $a^{ij}$ are measurable in $(t,x^1,x^2)$ and has small mean oscillation in $x^3$.

\section{Proof of Theorem \ref{thm2}}
                                        \label{sec3}

We prove Theorem \ref{thm2} in this section. Let
$$
P_0u=-u_t+a^{ij}D_{ij}u,
$$
where $a^{11}=a^{11}(t)$ and $a^{ij}=a^{ij}(t,x^1)$ for $ij>1$.

Our proof relies mainly on the following solvability result.
\begin{theorem}
						\label{thm2.2}
Let $p\in (1,\infty)$ and $T\in (-\infty,\infty]$. Then for any $u\in W^{1,2}_p(\bR^{d+1}_T)$ and $\lambda\ge 0$, we have
\begin{align}
\lambda\|u\|_{L_{p}(\bR^{d+1}_T)}+\sqrt{\lambda}
\|Du\|_{L_{p}(\bR^{d+1}_T)}&+\|D^2u\|_{L_{p}(\bR^{d+1}_T)}
+\|u_{t}\|_{L_{p}(\bR^{d+1}_T)}\nonumber\\
                \label{eq10.30}
\le &N\|P_0u-\lambda u\|_{L_{p}(\bR^{d+1}_T)},
\end{align}
where $N=N(d,p,\delta)>0$.
Moreover, for any $f\in L_p(\bR^{d+1}_T)$ and $\lambda>0$ there is a unique $u\in W^{1,2}_p(\bR^{d+1}_T)$ solving
$$
P_0 u-\lambda  u=f \quad\text{in}\,\,\,\bR^{d+1}_T.$$
\end{theorem}
\begin{proof}
First we assume $T=\infty$. By the method of continuity, it suffices to prove the a priori estimate \eqref{eq10.30}. Also we can assume $u\in C_0^\infty$ by a density argument. Let $f=P_0 u-\lambda u$. Note that
\begin{equation}
 				\label{eq2.32}
-u_t+a^{11}D_1^2 u+\Delta_{d-1}u-\lambda u=f+\sum_{ij>1}(\delta_{ij}-a^{ij})D_{ij}u.
\end{equation}
By using Theorem 5.5 of \cite{Krylov_2007_mixed_VMO}, the left-hand side of \eqref{eq10.30} is less than
$$
N\|f\|_{L_{p}}+N\sum_{ij>1}\|D_{ij}u\|_{L_{p}}.
$$
Thus, we only need to prove
\begin{equation}
                \label{eq10.36}
\|D_{xx'}u\|_{L_{p}}\le N\|f\|_{L_{p}}.
\end{equation}

Rewrite operator $P_0$ into divergence form
$$
P_0 u=-u_t+D_1(a^{11} D_1 u)+\sum_{j=2}^d D_{j}\left((a^{1j}+a^{j1})D_1u
\right)+\sum_{i,j=2}^d D_{j}(a^{ij}D_iu).
$$
Clearly, for $k=2,...,d$, $D_k u$ satisfies
$$
P_0(D_k u)-\lambda D_k u=D_k f.
$$
Owing to Proposition \ref{lemA3}, we have the $\cH^{1}_p$-estimate of divergence form parabolic equations with leading coefficients under the condition of the theorem (and $a^{ij}$ not necessary to be symmetric). Therefore, we get
\begin{equation*}
\|D_{xx^k} u\|_{L_p}\le N\|f\|_{L_p}.
\end{equation*}
This gives \eqref{eq10.36} and completes the proof of the theorem when $T=\infty$.

For general $T\in (-\infty,\infty]$, we use the fact that $u=w$ for $t<T$, where $w\in
W_p^{1,2}$ solves
$$
P_0 w-\lambda w=\chi_{t<T}(P_0 u-\lambda u).
$$ The theorem is proved.
\end{proof}

An immediate corollary of Theorem \ref{thm2.2} is the following solvability result of the initial value problem.

\begin{corollary}
						\label{cor2.3}
Let $p\in (1,\infty)$ and $T\in (0,\infty)$. Then for any $f\in L_p((0,T)\times\bR^{d})$ and $\lambda>0$ there is a unique $u\in W^{1,2}_p((0,T)\times\bR^{d})$ solving $P_0 u-\lambda  u=f$ in $(0,T)\times\bR^{d}$ and $u(0,\cdot)=0$. Moreover, we have
$$
\lambda\|u\|_{L_{p}((0,T)\times\bR^{d})}+\sqrt{\lambda}
\|Du\|_{L_{p}((0,T)\times\bR^{d})}+\|D^2u\|_{L_{p}
((0,T)\times\bR^{d})}
$$
\begin{equation*}
+\|u_{t}\|_{L_{p}((0,T)\times\bR^{d})}
\le N\|P_0u-\lambda u\|_{L_{p}((0,T)\times\bR^{d})},
\end{equation*}
where $N=N(d,p,\delta)>0$.
\end{corollary}

To prove the solvability for general operator $P$, we need the following H\"older estimate, which is an immediate consequence of the Krylov--Safonov estimate.
\begin{lemma}
                                     \label{lemma 9.10.1}
Let $p\in (1,\infty)$, $d\geq 2$, $\kappa \ge 2$, and $r > 0$.
Assume that $u \in
C_0^{\infty}$  and $P_0 u =0$ in $Q_{\kappa r}$.
Then there exist constants
$N = N(d,p, \delta)$ and $\alpha=\alpha(d,p,\delta)\in (0,1]$ such that
for any multi-index  $\gamma=(\gamma^1,\gamma')$,
\begin{equation*}
\dashint_{Q_r} |D^{\gamma'}u  -
 (D^{\gamma'}u)_{Q_r}|^p \, dx\,dt
\le N \kappa^{-p\alpha} \left(|D^{\gamma'}u|^{p}\right)_{Q_{\kappa r}}.
\end{equation*}
\end{lemma}

\begin{proof} Since $LD^{\gamma'}u=0$ in $Q_{\kappa r}$,
it suffices to concentrate on $\gamma=0$. By using scaling
we reduce the general situation to the one in which $r=1$.
By Lemma 4.2.4 of \cite{Kr85} and Theorem 7.21 of \cite{Li},
$$
\osc _{Q_{1/\kappa}}u
\le N\kappa^{-\alpha} \|u \|_{L_p(Q_1)}
$$
with $\alpha$ and $N$ as in the statement. Scaling this estimate
shows that
$$
\osc _{Q_{1}}u
\le N \kappa^{-\alpha}\left(| u|^{p}\right)_{Q_{\kappa }}^{1/p}.
$$
It only remains to observe that
$$
\dashint_{Q_1} | u  -
 ( u)_{Q_1}|^p \, dx\,dt\leq N
 (\osc _{Q_{1}}u)^p.
$$
The lemma is proved.
\end{proof}

\begin{theorem}
                                        \label{thm12.02}
Let $d\geq 2$, $p\in (1,\infty)$ and let
$\alpha$ be the constant in Lemma \ref{lemma 9.10.1}. Then there is a constant $N$ depending only on $d,p,\delta$ such that for any $u\in W^{1,2}_{p,\text{loc}}$, $r\in (0,\infty)$, and
$\kappa\ge 4$,
\begin{equation*}
\left(|D_{x'}^2u(t,x)-(D_{x'}^2 u)_{Q_r}|^p\right)_{Q_r}
\le N\kappa^{d+2}
\left(|P_0u|^p\right)_{Q_{\kappa
r}}+N\kappa^{-p\alpha}
\left(|D_{x'}^2 u|^p\right)_{Q_{\kappa r}}.
\end{equation*}
\end{theorem}
\begin{proof}
The theorem follows from Corollary \ref{cor2.3} and Lemma \ref{lemma 9.10.1}; see, for instance, the proof of Theorem 4.5 \cite{DongKrylov}.
\end{proof}

Next we generalize Theorem \ref{thm12.02}.
\begin{theorem}
                                            \label{thm4.1}
Let $d\geq 2$, $p\in (1,\infty)$ and let
$\alpha$ be the constant in
Lemma \ref{lemma 9.10.1}, $\gamma > 0$, $\tau,\sigma \in (1,\infty)$,
$1/\tau+1/\sigma=1$. Assume $b^i=c=0$ and $u\in W^{1,2}_p$. Then under
Assumption \ref{assump2} ($\gamma$) i) and iii) there exists a positive constant $N$ depending only on
$d$, $p$, $\delta$, and $\tau$ such that, for any
$(t_0,x_0)\in \bR^{d+1}$,
$r\in (0,\infty)$, and $\kappa\ge 4$,
\begin{multline}
                                \label{eq13.5.05}
\left(|D_{x'}^2 u(t,x)-(D_{x'}^2u)_{Q_r(t_0,x_0)}|^p
\right)_{Q_r(t_0,x_0)}\\
\le N\kappa^{d+2}
\left(|Pu|^p\right)_{Q_{\kappa r}(t_0,x_0)}
+N\kappa^{d+2}\gamma^{1/\sigma}
\left(|D^2 u|^{p\tau}\right)_{Q_{\kappa r}(t_0,x_0)}^{1/\tau}\\
+N\kappa^{-p\alpha}
\left(|D_{x'}^2 u|^p\right)_{Q_{\kappa r}(t_0,x_0)},
\end{multline}
provided that $u$ vanishes outside  $Q_{R_0}$.
\end{theorem}
\begin{proof}
Set $f=P u$.
We fix $(t_0,x_0)\in \bR^{d+1}$, $\kappa\ge 4$, and $r\in
(0,\infty)$.
  Choose $Q$ to be  $Q_{\kappa
r}(t_0,x_0)$ if $\kappa r< R_0$ and   $Q_{ R_0 }$ if $\kappa
r\ge  R_0 $. Recall the definitions of
$\bar{a}^{11}(t)$ and $\bar{a}^{ij}(t)$ given before Assumption \ref{assump2} i) and iii), respectively. We set
$$
\hat{f}= -u_{t}+\bar a^{ij}D_{ij}u.
$$
by Theorem \ref{thm12.02} with an appropriate translation  and $\bar a$ in
place of $a$,
\begin{multline}							
                                \label{13.5.42}
\dashint_{Q_r(t_0,x_0)} | D_{x'}^2 u
- \left( D_{x'}^2 u \right)_{Q_r(t_0,x_0)} |^p \, dx\,dt\\
\le N\kappa^{d+2}
\left(| \hat f|^p\right)_{Q_{\kappa r}(t_0,x_0)}
+N\kappa^{-p\alpha}
\left(|D_{x'}^2 u|^p\right)_{Q_{\kappa r}(t_0,x_0)},
\end{multline}
where $N$ depends only on $d$, $p$ and $\delta$.
By the definition of $\hat f$,
\begin{equation}							\label{13.5.52}
\int_{Q_{\kappa r}(t_0,x_0)} |\hat{f}|^p \, dx\,dt
\le N \int_{Q_{\kappa r}(t_0,x_0)} |f|^p \, dx\,dt
+N I,
\end{equation}
where
\begin{align*}
I &=
\int_{Q_{\kappa r}(t_0,x_0)}
\big| (\bar a^{ij} - a^{ij}) D_{ij}u \big|^p \, dx\,dt\\
&= \int_{Q_{\kappa r}(t_0,x_0) \cap Q_{ R_0 }}
\big| (\bar a^{ij} - a^{ij}) D_{ij}u \big|^p \, dx\,dt.
\end{align*}
By H\"older's inequality, we have
\begin{equation}							\label{13.5.53}
I \le NI_1^{1/\sigma} I_2^{1/\tau},
\end{equation}
where
\begin{align*}
I_1 &= \sum_{i,j}
\int_{Q_{\kappa r}(t_0,x_0) \cap Q_{ R_0}}
| \bar a^{ij} - a^{ij} |^{p\sigma} \, dx\,dt,\\
I_2 &= \int_{Q_{\kappa r}(t_0,x_0)} |D^2 u|^{p\tau} \, dx\,dt.
\end{align*}
Due to Assumption \ref{assump2} ($\gamma$) i) and iii),
$$
I_1 \leq \sum_{i,j}\int_{Q }
| \bar a^{ij} - a^{ij} |^{p\sigma} \, dx\,dt\leq N\gamma|Q|
\le  N  (\kappa r)^{d+2} \gamma,
$$
where $|Q|$ is the volume of $Q$.
This together with  \eqref{13.5.42}-\eqref{13.5.53} yields
\eqref{eq13.5.05}. The theorem is proved.
\end{proof}

\begin{lemma}
 						\label{lem2.1}
Assume $b^i=c=0$. For any
$p\in (1,\infty)$ there exists a constant $\gamma_0=\gamma_0(d,\delta,p) >0$
such that under Assumption
\ref{assump2} ($\gamma_0$) i) for any $u\in C_0^\infty(Q_{R_0})$, we have
\begin{equation}
                \label{eq11.27}
\|D^2 u\|_{L_{p}}+\|u_{t}\|_{L_{p}}\le N\|Pu\|_{L_{p}}+N\|D_{x'}^2 u\|_{L_{p}},
\end{equation}
where $N=N(d,\delta,p)$. In particular, in case $d=1$, we have
$$
\|D^2 u\|_{L_{p}}+\|u_{t}\|_{L_{p}}\le N\|Pu\|_{L_{p}}.
$$
\end{lemma}
\begin{proof}
We write
$$
-u_t+a^{11}D_1^2 u+\Delta_{d-1}u=Pu+\sum_{ij>1}(\delta_{ij}-a^{ij})D_{ij} u.
$$
The coefficients on the left-hand side above was studied in \cite{Krylov_2005}.
By Corollary 3.7 \cite{Krylov_2005}, it holds that
\begin{equation*}
\|D^2 u\|_{L_{p}}+\|u_{t}\|_{L_{p}}\le N\|Pu\|_{L_{p}}+N\|D_{xx'} u\|_{L_{p}}.
\end{equation*}
Finally, to conclude \eqref{eq11.27} it suffices to notice that
for any $\varepsilon>0$,
\begin{equation}
                                                    \label{3.57}
\|D_{x^1x'}u\|_{L_{p}}\le \varepsilon\|D_1^2 u\|_{L_{p}}
+N(d,p)\varepsilon^{-1}\|D_{x'}^2 u\|_{L_{p}},
\end{equation}
 which is deduced from
$$
\|D_{x^1x'}u\|_{L_{p}}\le N
\|\Delta u\|_{L_{p}}\leq N\|D_{1}^2 u\|_{L_{p}}+N \|D_{x'}^2 u\|_{L_{p}}
$$
by scaling in $x^1$.
\end{proof}

\begin{lemma}
                                                 \label{lem4.2}
Let $p\in (1,\infty)$ and $f\in L_p$. Assume $b^i=c=0$. Then there exist positive constants
$\gamma$  and $N$ depending only on $d,p$ and $\delta$ such that under Assumption \ref{assump2} ($\gamma$) i) and iii), for any $u\in W^{1,2}_p$
vanishing outside $Q_{ R_0}$   and satisfying
$Pu =f$, we have
$$
\|u_t\|_{L_p}+\|D^2u\|_{L_p}\le N\|f\|_{L_p}.
$$
\end{lemma}
\begin{proof}
The case $d=1$ follows from Lemma \ref{lem2.1}. In the sequel, we assume $d\ge 2$. We fix two numbers $q=(1+p)/2\in (1,p)$, $\tau=(1+3p)/(2+2p)\in (1,\infty)$ so that $p>\tau q$. Let $\alpha$ and $\gamma_0$ be the constants in Lemma \ref{lemma 9.10.1} and Lemma \ref{lem2.1} respectively.
Let $\gamma\in (0,\gamma_0)$ be a number to be specified later.
Inequality \eqref{eq13.5.05} with $q$ in place of $p$ implies that
on $\bR^{d+1}$
\begin{multline*}
(D_{x'}^2 u)^{\#}
\le N \kappa^{(d+2)/q} \bM^{1/q}(|f|^q)\\
+  N \kappa^{(d+2)/q}\gamma^{1/(q\sigma)} \bM^{1/(q\tau)}(|D^2 u|^{q\tau})
+N\kappa^{-\alpha}\bM^{1/q}(|D^2 u|^{q}).
\end{multline*}
We apply the Fefferman--Stein theorem on sharp functions, which
reads that for $p\in (1,\infty)$ the $L_p$ norm of a function is
bounded by the $L_p$ norm of its sharp function, and the
Hardy--Littlewood maximal function theorem, which reads that the
$L_p$ norm of the maximal function is bounded by the $L_p$ norm of
the function itself.
We then get
\begin{multline*}
\|D^2_{x'} u\|_{L_p}
\le N \| (D^2_{x'} u)^{\#} \|_{L_p}
\le N \kappa^{(d+2)/q} \| \bM(|f|^q) \|_{L_{p/q}}^{1/q}\\
+ N\kappa^{(d+2)/q}\gamma^{1/(q\sigma)} \|\bM(|D^2 u|^{q\tau})
\|_{L_{p/(q\tau)}}^{1/(q\tau)} +N\kappa^{-\alpha}\| \bM(|D^2 u|^q)
\|_{L_{p/q}}^{1/q}\\
\le N \kappa^{(d+2)/q} \| f \|_{L_p} + N \left(\kappa^{(d+2)/q}
\gamma^{1/(q\sigma)}+\kappa^{-\alpha}\right)
\|D^2 u\|_{L_{p}},
\end{multline*}
where in the last inequality we use the fact that $p > \tau q$.
From this estimate and Lemma \ref{lem2.1}, we have
$$
\|D^2 u\|_{L_p}+\|u_{t}\|_{L_p}
\le N \kappa^{(d+2)/q} \| f \|_{L_p}
 + N \left(\kappa^{(d+2)/q}
\gamma^{1/(q\sigma)}+\kappa^{-\alpha}\right)
\|D^2 u\|_{L_{p}} .
$$
To finish the proof of the lemma, it suffices to choose a large $\kappa$ and then a small $\gamma$
so that
$$
N \left(\kappa^{(d+2)/q} \gamma^{1/(q\sigma)}+\kappa^{-\alpha}\right) \le
1/2.
$$
\end{proof}

Now we are ready to prove Theorem \ref{thm2}.
\begin{proof}[Proof of Theorem \ref{thm2}]
To prove i), as in the proof of Theorem \ref{thm2.2}, we only need to verify \eqref{13.10.21} for $T=\infty$ and $u\in C_0^\infty$. This in turn is
obtained from Lemma
\ref{lem4.2} by using a partition of unity and
an idea of Agmon (see, for instance,  \cite{Krylov_2005}). Assertion ii) is obtained from i) by the method of continuity, and Assertion iii) is already proved in Theorem \ref{thm2.2}. The theorem is proved.
\end{proof}

\section{Proof of Theorem \ref{thm3}}
                                \label{sec4}

The objective of this section is to prove Theorem \ref{thm3}. We follow the strategy in the previous section. However, the proofs here are more involved.

Let
$$
P_0u=-u_t+a^{ij}D_{ij}u,
$$
where
$$
a^{11}=a^{11}(x^1),\quad a^{ij}=a^{ij}(t,x^1)\quad \text{for}\,\,\,ij>1.$$
Similarly we need a solvability result for equations with these simple leading coefficients.
\begin{theorem}
						\label{thm3.2}
Let $p\in (1,\infty)$ and $T\in (-\infty,\infty]$. Then for any $u\in W^{1,2}_p(\bR^{d+1}_T)$ and $\lambda\ge 0$, we have
$$
\lambda\|u\|_{L_{p}(\bR^{d+1}_T)}+\sqrt{\lambda}
\|Du\|_{L_{p}(\bR^{d+1}_T)}+\|D^2u\|_{L_{p}(\bR^{d+1}_T)}
+\|u_{t}\|_{L_{p}(\bR^{d+1}_T)}
$$
\begin{equation}
                \label{eq10.30b}
\le N\|P_0u-\lambda u\|_{L_{p}(\bR^{d+1}_T)},
\end{equation}
where $N=N(d,p,\delta)>0$.
Moreover, for any $f\in L_p(\bR^{d+1}_T)$ and $\lambda>0$ there is a unique $u\in W^{1,2}_p(\bR^{d+1}_T)$ solving
$$
P_0 u-\lambda  u=f \quad\text{in}\,\,\,
\bR^{d+1}_T.
$$
\end{theorem}
\begin{proof}
As in the proof of Theorem \ref{thm2.2}, it suffices to prove the a priori estimate \eqref{eq10.30b} for $u\in C_0^\infty$ and $T=\infty$. Let $\lambda>0$ and $f=P_0 u-\lambda u$. Here we are not able to directly write $P_0$ into a divergence form operator. However, this can be done after a change of variables:
$$
y^1=\phi(x^1):=\int_0^{x^1} \frac 1 {a^{11}(s)}\,ds,\quad y^j=x^j,\,\,j\ge 2.
$$
It is easy to see that $\phi$ is a bi-Lipschitz map and
\begin{equation*}
\delta \le y^1/x^1\le \delta^{-1},\quad D_{y^1}=a^{11}(x^1)D_{x^1}.
\end{equation*}
Denote
\begin{align*}
v(t,y^1,y')&=u(t,\phi^{-1}(y^1),y'),\\
\bar a^{11}(y^1)&=a^{11}(\phi^{-1}(y^1)),\\
\bar a^{ij}(t,y^1)&=a^{ij}(t,\phi^{-1}(y^1)),\,\,\,ij>1,\\
\bar f(t,y)&=f(t,\phi^{-1}(y^1),y').
\end{align*}
In the $(t,y)$-coordinates, define a divergence form operator $\bar P_0$ by
$$
\bar P_0 v=-v_t+D_{1}\left(\frac 1 {\bar a^{11}} D_1  v\right)+\sum_{j=2}^d D_{j}\left(\frac {\bar a^{1j}+\bar a^{j1}}
{\bar a^{11}} D_1v
\right)+\sum_{i,j=2}^d D_{j}(\bar a^{ij}D_iv).
$$
It is easily seen that $\bar P_0$ is uniformly nondegenerate with an ellipticity constant depending only on $\delta$. A simple calculation shows that $v$ satisfies in $\bR^{d+1}$
$$
\bar P_0 v-\lambda v=\bar f.
$$
By Proposition \ref{lemA1}, we have
\begin{equation*}
\lambda \| v \|_{L_p}
+ \sqrt\lambda\|Dv\|_{L_p}
\le N\|\bar f\|_{L_p}.
\end{equation*}
Therefore,
\begin{equation}
					\label{eq2.25}
\lambda \|u \|_{L_p}
+ \sqrt\lambda\|Du\|_{L_p}
\le N\|f\|_{L_p}.
\end{equation}
Next we estimate $D^2 u$. Notice that for each $k=2,...,d$ $D_kv$ satisfies
$$
\bar P_0 (D_k v)-\lambda D_k v=D_k \bar f.
$$
Again by using Proposition \ref{lemA1}, we get
$$
\| D_{yy^k} v \|_{L_p}
\le N\|\bar f\|_{L_p},
$$
which implies
\begin{equation}
				\label{eq2.57}
\| D_{xx'} u \|_{L_p}
\le N\|f\|_{L_p}.
\end{equation}
Finally, to estimate $D_1^2 u$, we return to the equation in the original coordinate system. From \eqref{eq2.32}, we see that $w:=D_1 u$ satisfies
\begin{equation}
						\label{eq2.55}
-w_t+D_1(a^{11}D_1w)+\Delta_{d-1}w-\lambda w=D_1 f+\sum_{ij>1}D_1
\left((\delta_{ij}-a^{ij})D_{ij}u\right).
\end{equation}
We use Proposition \ref{lemA1} again to get
\begin{equation}
				\label{eq2.36}
\| D_{1}^2 u \|_{L_p}\le \| Dw \|_{L_p}\le N\| f \|_{L_p}
+N\sum_{ij>1} \| D_{ij}u \|_{L_p}.
\end{equation}
Combining \eqref{eq2.25}, \eqref{eq2.57} and \eqref{eq2.36} yields \eqref{eq10.30b} by bearing in mind that
\begin{equation}
					\label{eq3.03}
u_t=a^{ij}D_{ij}u-\lambda u-f.
 \end{equation}
The theorem is proved.
\end{proof}

Observe that in the proofs of Lemma \ref{lem2.1}, Theorems \ref{thm12.02} and \ref{thm4.1}, we only used the solvability of parabolic equations with ``simple'' leading coefficients and the fact that $a^{ij}$ are independent of $x'$. Since Theorem \ref{thm3.2} is proved, we still have \eqref{eq13.5.05} if Assumption \ref{assump2} i) in Theorem \ref{thm4.1} is replace by Assumption \ref{assump2} ii) . More precisely,

\begin{theorem}
                                            \label{thm5.2}
Let $d\geq 2$, $p\in (1,\infty)$ and let
$\alpha$ be the constant in
Lemma \ref{lemma 9.10.1}, $\gamma > 0$, $\tau,\sigma \in (1,\infty)$,
$1/\tau+1/\sigma=1$. Assume $b^i=c=0$ and $u\in W^{1,2}_p$. Then under
Assumption \ref{assump2} ($\gamma$) ii) and iii) there exists a positive constant $N$ depending only on
$d$, $p$, $\delta$, and $\tau$ such that, for any
$(t_0,x_0)\in \bR^{d+1}$,
$r\in (0,\infty)$, and $\kappa\ge 4$, we have
\begin{multline*}
\left(|D_{x'}^2 u(t,x)-(D_{x'}^2u)_{Q_r(t_0,x_0)}|^p
\right)_{Q_r(t_0,x_0)}
\le N\kappa^{d+2}
\left(|Pu|^p\right)_{Q_{\kappa r}(t_0,x_0)}\\
+N\kappa^{d+2}\gamma^{1/\sigma}
\left(|D^2 u|^{p\tau}\right)_{Q_{\kappa r}(t_0,x_0)}^{1/\tau}
+N\kappa^{-p\alpha}
\left(|D_{x'}^2 u|^p\right)_{Q_{\kappa r}(t_0,x_0)},
\end{multline*}
provided that $u$ vanishes outside  $Q_{R_0}$.
\end{theorem}

\begin{lemma}
 						\label{lem5.2}
Assume $b^i=c=0$. For any
$p\in (1,\infty)$ there exists a $\gamma_1$, $\mu_1$ and $N$, depending only on $d$, $\delta$ and $p$, such that under Assumption
\ref{assump2} ($\gamma_1$) ii) for any $u\in C_0^\infty(Q_{\mu_1^{-1}R_0})$, we have
\begin{equation}
                \label{eq2.53}
\|D^2 u\|_{L_{p}}+\|u_{t}\|_{L_{p}}\le N\|Pu\|_{L_{p}}+N\|D_{x'}^2 u\|_{L_{p}}.
\end{equation}
\end{lemma}
\begin{proof}
Because $w=D_1 u$ satisfies \eqref{eq2.55}, due to Proposition \ref{lemA2} we have
$$
\|D_1^2 u\|_{L_p}\le \|Dw\|_{L_p}\le N\|f\|_{L_p}+\sum_{ij>1}\|D_{ij}u\|_{L_p}.
$$
This and \eqref{eq3.03} with $\lambda=0$ yield
$$
\|u_t\|_{L_p}+\|D^2 u\|_{L_p}\le N\|f\|_{L_p}+\sum_{ij>1}\|D_{ij}u\|_{L_p}.
$$
To finish the proof of \eqref{eq2.53}, it suffices to use \eqref{3.57}.
\end{proof}

By using Theorem \ref{thm5.2} and Lemma \ref{lem5.2}, we can prove Theorem \ref{thm3} in exactly the same way as in the proof of Theorem \ref{thm2}. We omit the details.

\section{An extension to parabolic systems}  		 \label{sec5}

The objective of this section is to extend Theorem \ref{thm2} to parabolic systems under the same regularity assumption on the leading coefficients. Let $m\ge 2$ be an integer. We consider parabolic operators in nondivergence form
\begin{equation*}                          
P \vu = -\vu_t + A^{\alpha \beta}
D_{\alpha\beta}\vu + B^{\alpha}
D_\alpha\vu + C \vu,
\end{equation*}
acting on (column) vector-valued functions
$\vu = (u^1, \ldots, u^m)^T$. Here for each $\alpha,\beta=1,...,d$, $A^{\alpha\beta}$, $B^\alpha$ and $C$ are $m\times m$ matrix-valued functions
given on $\bR^{d+1}$;
i.e.,
$A^{\alpha\beta}=[A^{\alpha \beta}_{ij}(t,x)]_{m \times m}$, etc.

As in the scalar case, all the coefficients are assumed to be bounded and measurable, and $A^{\alpha\beta}$ are uniformly elliptic, i.e.
\begin{equation}
                            \label{eq15.41}
\begin{split}
|B^\alpha|+|C|\le K &,\quad |A_{ij}^{\alpha\beta}|\le \delta^{-1},\\
\delta \sum_{i=1}^{m}\sum_{\alpha=1}^d |\xi_{\alpha}^i|^2\le
\sum_{\alpha,\beta=1}^d \sum_{i, j =1}^{m}
A_{ij}^{\alpha \beta}(&t,x) \xi^i_{\alpha} \xi^j_{\beta}
\le \delta^{-1} \sum_{i=1}^{m}\sum_{\alpha=1}^d |\xi_{\alpha}^i|^2
\end{split}
\end{equation}
for all $(t,x) \in \bR^{d+1}$ and $\xi_{\alpha} = (\xi^i_{\alpha}) \in \bR^m$, $\alpha = 1, \ldots, d$. Additionally, we assume $A^{\beta 1}$ is symmetric matrix for $\beta=2,\ldots,d$.

Similarly we also define the objects $\bar A^{11}$, $\bar A^{ij},ij>1$, $A_{R,1}^{11,\#}$ and $A_{R}^{\#}$ as in Section \ref{mainsec}.

\begin{assumption}[$\gamma$]
                                            \label{assump5}
There exists a positive constant $R_0$ such that
$$
A^{11,\#}_{R_0,1}+A^{\#}_{R_0}\le \gamma.
$$
\end{assumption}

Denote
$$
P_0\vu=-\vu_t+A^{\alpha\beta}D_{\alpha\beta}\vu,
$$
where
$$A^{11}=A^{11}(t),\quad A^{\alpha\beta}=A^{\alpha\beta}(t,x^1),\,\,\alpha\beta>1.
$$
First we give a counterpart of Theorem \ref{thm2.2}.

\begin{theorem}
						\label{thm4.2}
Let $p\in (1,\infty)$ and $T\in (-\infty,\infty]$. Then for any $\vu\in W^{1,2}_p(\bR^{d+1}_T)$ and $\lambda\ge 0$, we have
$$
\lambda\|\vu\|_{L_{p}(\bR^{d+1}_T)}+\sqrt{\lambda}
\|D\vu\|_{L_{p}(\bR^{d+1}_T)}+\|D^2\vu\|_{L_{p}(\bR^{d+1}_T)}
+\|\vu_{t}\|_{L_{p}(\bR^{d+1}_T)}
$$
\begin{equation*}
\le N\|P_0\vu-\lambda \vu\|_{L_{p}(\bR^{d+1}_T)},
\end{equation*}
where $N=N(d,m,p,\delta)>0$.
Moreover, for any $\vf\in L_p(\bR^{d+1}_T)$ and $\lambda>0$ there is a unique $\vu\in W^{1,2}_p(\bR^{d+1}_T)$ solving
$$
P_0 \vu-\lambda  \vu=\vf \quad\text{in}\,\,\,\bR^{d+1}_T.
$$
\end{theorem}
\begin{proof}
The proof is almost the same as that of Theorem \ref{thm2.2}. So we only give a few comments. The only difference is that instead of Theorem 5.5 of \cite{Krylov_2007_mixed_VMO}, we use a system version of it, which can be found in Theorem 3.1 of \cite{DK08}.
Here the symmetric condition of $A^{\beta 1},\beta=2,\ldots,d$ is needed to apply Proposition \ref{lemA3}.
\end{proof}

\begin{remark}
 				\label{rem4.32}
The results in \cite{KK2,Kim07,Kim07a,Kim07b} rely on the $W^{1,2}_2$-solvability proved in \cite{KK2} for parabolic equations with $a^{ij}$ depending only on $(t,x^1)$. The extension of the latter to parabolic systems is unknown to us, except for the one space dimensional case. So is the corresponding extension of $W^{2}_2$-solvability to elliptic systems. This is because the proof in \cite{KK2} uses the maximum principle, which is only valid for scalar equations.
\end{remark}

The key estimate in this section is the following lemma, which is a generalization of Lemma \ref{lemma 9.10.1}. We note that, since we are dealing with systems, the Krylov--Safonov estimate does not apply.

\begin{lemma}
                                     \label{lem4.4}
Let $p\in (1,\infty)$, $d\geq 2$, $\kappa \ge 2$, and $r > 0$.
Assume that $\vu \in
C_0^{\infty}$  and $P_0 \vu =0$ in $Q_{\kappa r}$.
Then for any $\alpha\in (0,1)$, there exist constants
$N = N(d,m,p, \delta,\alpha)$ such that
for any multi-index  $\gamma=(\gamma^1,\gamma')$
\begin{equation}
                                          \label{9.10.1}
\dashint_{Q_r} |D^{\gamma'}\vu  -
 (D^{\gamma'}\vu)_{Q_r}|^p \, dx\,dt
\le N \kappa^{-p\alpha} \left(|D^{\gamma'}\vu|^{p}\right)_{Q_{\kappa r}}.
\end{equation}
\end{lemma}
\begin{proof}
As in the proof of Lemma \ref{lemma 9.10.1}, without loose of generality, we may assume $|\gamma|=0$. We again use the fact that $P_0\vu$ can be rewritten into a divergence form:
$$
-\vu_t+D_1(A^{11} D_1 \vu)+\sum_{\beta=2}^d D_{\beta}\left((A^{1\beta}+A^{\beta 1})D_1\vu
\right)+\sum_{\alpha,\beta=2}^d D_{\beta}(A^{\alpha\beta}D_\alpha \vu).
$$
Now \eqref{9.10.1} follows from Proposition \ref{lemA4}. The lemma is proved.
\end{proof}

\begin{remark}
We claim a more general estimate:  under the assumptions of Lemma \ref{lem4.4}, we have
$$
\dashint_{Q_r} |D^{\gamma}\vu  -
 (D^{\gamma}\vu)_{Q_r}|^p \, dx\,dt
\le N \kappa^{-p\alpha} \left(|D^{|\gamma^1|}D^{\gamma'}\vu|^{p}\right)_{Q_{\kappa r}}.
$$
provided that $\gamma^1=0$ or $1$. Indeed, it suffices to consider the case $\gamma^1=1$ and $|\gamma'|=0$. This case can be proved by using the iteration argument in the proof of Lemma 4.2 \cite{DongKim08b} and the Poincar\'e's inequality. We leave the details to the interested reader. This estimate will not be used in the sequel.
\end{remark}

With Theorem \ref{thm4.2} and Lemma \ref{lem4.4} available, we can obtain analogous results of Theorems \ref{thm12.02} and \ref{thm4.1}. Following the lines of Section \ref{sec3}, we get the following solvability result, which is a generalization of Theorem \ref{thm2}.

\begin{theorem}
                                         \label{thm4.6}

For any $p\in (1,\infty)$,
there exists a $\gamma =\gamma(d,m,\delta,p) >0$
such that under Assumption
\ref{assump5} ($\gamma$) for any $T\in (-\infty,+\infty]$ the following holds.

i) For any $\vu\in W^{1,2}_p(\bR^{d+1}_T)$,
$$
\lambda\|\vu\|_{L_{p}(\bR^{d+1}_T)}+\sqrt{\lambda}
\|D\vu\|_{L_{p}(\bR^{d+1}_T)}+\|D^{2}\vu\|_{L_{p}(\bR^{d+1}_T)}
+\|\vu_t\|_{L_{p}(\bR^{d+1}_T)}
$$
\begin{equation*}
\leq N\|P\vu-\lambda \vu\|_{L_{p}(\bR^{d+1}_T)},
\end{equation*}
provided that $\lambda\geq \lambda_0$, where $\lambda_0
\geq0$ depends only
 on $d,m,\delta,p,K$ and $R_0$, and $N$ depends only on $d,m,\delta$ and $p$.

ii) For any $\lambda> \lambda_0$ and $\vf\in L_p(\bR^{d+1}_T)$, there exists a unique solution
$\vu\in W^{1,2}_p(\bR^{d+1}_T)$ of
$$
P\vu-\lambda\vu=\vf \quad \text{in}\,\,\, \bR^{d+1}_T.
$$

iii) In the case that
$A^{11}=A^{11}(t)$, $A^{\alpha\beta}=A^{\alpha\beta}(t,x^1),
\alpha\beta>1$ and $b^j\equiv c\equiv 0$, we can take
$\lambda_0=0$ in i) and ii).
\end{theorem}

We remark that Theorem \ref{thm4.6} is new even when $p>2$. It is also worth noting that under Assumption \ref{assump5} ($\gamma$) solutions of the parabolic systems are H\"older continuous if $p>(d+2)/2$. Unlike scalar equations, for which we have the Krylov--Safonov H\"older estimate, generally this property is not possessed by solutions to systems (see, for instance, \cite{MaPaSo00}).

\section{Hierarchically partially  BMO coefficients}			\label{sec6}

We consider in this section parabolic equations with more general coefficients. The assumption is that for $ij>1$ $a^{ij}$ are measurable in $(t,x^1,\ldots,x^{\pi_{ij}})$ and BMO in the other coordinates, where
$$
\pi_{ij}=\max(i,j)-1.
$$
In addition, we suppose $a^{11}$ is measurable in $t$ (or $x^1$) and BMO in the other coordinates. More precise assumptions are stated below. To the best of our knowledge, this class of coefficients has been considered before.

We recall the definition of $a_{R,1}^{11,\#}$ and $a_{R,2}^{11,\#}$ in the introduction. For $R>0$, we denote
$$
a^{\#}_R=
\sup_{(t_0,x_0)\in \bR^{d+1}} \sup_{r\le R}\sup_{(i,j)\neq (1,1)}\dashint_{Q_r(t_0,x_0)}|a^{ij}-\bar a^{ij}|\,dx\,dt,
$$
where for each $Q_r(t_0,x_0)$ and $(i,j)\neq (1,1)$,
$$
\bar a^{ij}=\bar a^{ij}(t,x^1,\ldots,x^{\pi_{ij}})
$$
$$
=\dashint_{B_r^{d-\pi_{ij}}(x_0^{\pi_{ij}+1},\ldots,x_0^d)}
a^{ij}(t,x^1,...,x^{\pi_{ij}},
y^{\pi_{ij}+1},\ldots,y^d)\,dy^{\pi_{ij}+1}\ldots dy^d.
$$
We impose either one of the following assumptions on $a^{ij}$.
\begin{assumption}[$\gamma$]
                                            \label{assump7}
There exists a positive constant $R_0$ such that
$$
a^{11,\#}_{R_0,1}+a_R^{\#}\le \gamma.
$$
\end{assumption}
\begin{assumption}[$\gamma$]
                                            \label{assump7b}
There exists a positive constant $R_0$ such that
$$
a^{11,\#}_{R_0,2}+a_R^{\#}\le \gamma.
$$
Clearly the assumptions above are weaker than those in Theorems \ref{thm2} and \ref{thm3} in terms of the regularity of $a^{ij}$ for $i>2$ or $j>2$.

Following the idea in Section \ref{sec3} and \ref{sec4}, to prove the $W^{1,2}_p$-solvability of nondivergence form equations under these assumptions, in Subsection \ref{subsec6.1} we first establish the corresponding $\cH^1_p$-solvability of divergence form equations under the same assumptions on the coefficients. In Subsection \ref{subsec6.2}, we prove the solvability for nondivergence equations.
\end{assumption}

\subsection{Solvability of divergence form equations} \label{subsec6.1}
Let us consider the $\cH^1_{p}$-solvability of parabolic equations in divergence form:
\begin{equation*}
\cP u-\lambda u=\Div g +f,
\end{equation*}
where $\lambda\ge 0$ is a constant and
\begin{align*}
\cP u&=-u_t+D_j(a^{ij}D_iu)+D_i(a^{i}u)+ b^{i}D_iu+cu,\\
g&=(g^1,g^2,\ldots,g^d).
\end{align*}

We assume the same boundedness and ellipticity condition on $a^{ij}$, $b^i$ and $c$ as in the introduction. We also assume that $a^i$ are measurable and bounded by $K$.

\begin{theorem}
                                            \label{thm5.1}
Let $p\in (1,\infty)$, there exists a constant $\gamma=\gamma(d,p,\delta)>0$
such that under Assumption \ref{assump7} ($\gamma$) (or  \ref{assump7b} ($\gamma$)) for any $T\in (-\infty,+\infty]$ the following holds.

i) For any $ u \in \cH_p^{1}(\bR^{d+1}_T)$,
\begin{equation*}
\lambda \|  u \|_{L_p(\bR^{d+1}_T)}
+ \sqrt{\lambda} \|Du\|_{L_p(\bR^{d+1}_T)}
\le N(\sqrt \lambda+1) \| \cP  u - \lambda  u \|_{\bH^{-1}_p(\bR^{d+1}_T)}
\end{equation*}
provided that $\lambda\ge \lambda_0$, where $\lambda_0\ge 0$ depends only
on $d$, $p$, $\delta$, $K$ and $R_0$,
and $N$ depends only on $d$, $p$ and $\delta$.

ii) For any $\lambda > \lambda_0$ and
$ f, g\in L_p(\bR^{d+1}_T)$,
there exists a unique $ u \in \cH_{p}^{1}(\bR^{d+1}_T)$ solving
$$\cP u - \lambda  u =  f+\Div g$$ in $\bR_T^{d+1}$. And $u$ satisfies the estimate
$$
\lambda \|  u \|_{L_p(\bR^{d+1}_T)}
+ \sqrt{\lambda} \|Du \|_{L_p(\bR^{d+1}_T)}
\le N\sqrt\lambda\|g\|_{L_p(\bR^{d+1}_T)} +N\|f\|_{L_p(\bR^{d+1}_T)}.
$$

iii) In the case that
$a^{11}=a^{11}(t)$ (or $a^{11}=a^{11}(x^1)$ respectively), and
$$
a^{ij}=a^{ij}(t,x^1,\ldots,x^{\pi_{ij}}),\,\,\,
ij>1,\quad
a^j=b^j\equiv c\equiv 0,
$$
we can take
$\lambda_0=0$ in i) and ii).
\end{theorem}

The theorem above generalized the main result of \cite{Dong08}. The idea in the proof is to ``break the symmetry'' of the coordinates.

For simplicity, we only give a proof of Theorem \ref{thm5.1} when $d=3$, or $d>3$ and $\pi_{ij}$ is replaced by $\tilde \pi_{ij}:=\min(\pi_{ij},2)$. The general case can be proved by an induction. We remark that the assumptions of Theorems \ref{thm2} and \ref{thm3} (and those in \cite{Dong08}, \cite{DongKim08b}) correspond to the case when $\pi_{ij}$ is replaced by $\min(\pi_{ij},1)$.

First let us consider the situation that Assumption \ref{assump7b} holds. We have the following corollary of Proposition \ref{lemA2}.

\begin{corollary}
				\label{cor5.5}
Let $p\in (1,\infty)$ and $g\in L_p$. Suppose $a^i\equiv b^i\equiv c\equiv 0$. Let $\mu_1$ and $\gamma_1$ be the constants in Proposition \ref{lemA2}. Then there exist  constants $\mu_2\in [\mu_1,\infty)$, $\gamma_2\in (0,\gamma_1]$ and $N$ depending only on $d,p$ and $\delta$
such that, under Assumption \ref{assump7b} ($\gamma_2$) with $\tilde \pi_{ij}$ in place of $\pi_{ij}$, for any $u\in C_0^\infty(Q_{\mu_2^{-1}R_0})$ satisfying $\cP u=\Div g$, we have
\begin{equation*}
\|Du\|_{L_p}\le N\|g\|_{L_p}+N\sum_{j=3}^d \|D_ju\|_{L_p}.
\end{equation*}
\end{corollary}
\begin{proof}
Let $\mu_2=\mu_1 \mu$ and $\gamma_2=\mu^{-4}\gamma_1$, where $\mu\in [1,\infty)$ is a number to be specified later. Denote
$$
\bar u(t,x)=u(\mu^{-2}t,\mu^{-1}x^1,\mu^{-1}x^2,x^3,\ldots,x^d).
$$ It is clear that $\bar u\in C_0^\infty(Q_{\mu_1^{-1}R_0})$, and $\bar u$ satisfies
$$
-\mu^2 \bar u_t+\sum_{i,j=1}^2 D_j(\mu^2 \bar a^{ij}D_i \bar u)+\sum_{i=1}^2\sum_{j>2}D_j(\mu \bar a^{ij} D_i\bar u)_{x^j}
$$
$$
+\sum_{j=1}^2\sum_{i>2}D_j (\mu \bar a^{ij} D_i \bar u)+\sum_{i>2,j>2}D_j (\bar a^{ij}D_i \bar u)=\Div(\bar g),
$$
where
\begin{align*}
\bar a^{ij}&=a^{ij}(\mu^{-2}t,\mu^{-1}x^1,\mu^{-1}x^2,x^3,\ldots,x^d),\\
\bar g^i&=\mu g^i( \mu^{-2}t,\mu^{-1}x^1,\mu^{-1}x^2,x^3,\ldots,x^d),\,\,i=1,2,\\
\bar g^i&=g^i( \mu^{-2}t,\mu^{-1}x^1,\mu^{-1}x^2,x^3,\ldots,x^d),\,\,i>2.
\end{align*}
Simple calculation shows that $\bar a^{ij}$ satisfy Assumption \ref{assump7b} ($\gamma_1$) with $\tilde \pi_{ij}$ in place of $\pi_{ij}$.

We define an operator $\mathfrak P$ by
$$
\mathfrak Pu=-u_t+\sum_{i,j=1}^2 D_j(\bar a^{ij} D_i u)+\sum_{j=3}^dD_j^2 u.
$$
Then we have
\begin{equation*}
\mathfrak P\bar u=\Div \mathfrak g,
\end{equation*}
where 
$$
\mathfrak g^k=\mu^{-2}\bar g^k-\mu^{-1}\sum_{i>2}\bar a^{ik}D_i \bar u,\,\,k=1,2,
$$
$$
\mathfrak g^k=\mu^{-2}\bar g^k-\mu^{-1}\sum_{i=1}^2 \bar a^{ik}D_i \bar u-
\mu^{-2}\sum_{i>2}\bar a^{ik}D_i \bar u+D_k \bar u,\,\,k\ge 3.
$$
Since the coefficients of $\mathfrak P$ satisfy Assumption \ref{assump2} ($\gamma_1$) ii) and iii), we get from Proposition \ref{lemA2} that
$$
\|D\bar u\|_{L_p}\le N_1\|\mathfrak g\|_{L_p},
$$
which implies
$$
\sum_{i=1}^2\|D_i u\|_{L_p}\le N_1\|g\|_{L_p}+N_1\mu^{-1}\sum_{i=1}^2\|D_i u\|_{L_p}+N_1\mu\sum_{i=3}^d\|D_i u\|_{L_p},
$$
where $N_1=N_1(d,p,\delta)>0$. To finish the proof, it suffices to choose $\mu=(2N_1)^{-1}$.
\end{proof}

Because of Corollary \ref{cor5.5}, we only need to obtain a good estimate of $\|D_j u\|_{L_p}$ for $j\ge 3$. Following the lines of Section 5 \cite{Dong08}, we have the next estimate of mean oscillations.

\begin{lemma}
                            \label{lem5.6}
Let $a^i=b^i=0$, $c=0$, $\sigma,\tau\in (1,\infty)$ satisfying $1/\sigma+1/\tau=1$. Assume $u\in \cH^1_{2,\text{loc}}$ and $\cP u=\Div g$, where $g\in L_{2,\text{loc}}$. Then under Assumption \ref{assump7b} ($\gamma$) with $\tilde \pi_{ij}$ in place of $\pi_{ij}$, there exist an $\alpha=\alpha(d,\delta)\in (0,1)$ and a positive constant $N$ depending only on $d$, $\sigma$ and $\delta$ such that
\begin{multline*}
\sum_{i=3}^d \left(|D_i u-(D_i u)_{Q_r(t_0,x_0)}|^2\right)_{Q_r(t_0,x_0)}\leq N\nu^{-2\alpha}
\left(|Du|^2\right)_{Q_{\nu r}(t_0,x_0)}\\
+N\nu^{d+2}\left( (|g|^2)_{Q_{\nu r}(t_0,x_0)}+\gamma^{1/\sigma}
(|Du|^{2\tau})_{Q_{\nu r}(t_0,x_0)}^{1/\tau}\right),
\end{multline*}
for any $r\in (0,\infty)$, $\nu\ge 4$ and $(t_0,x_0)\in \bR^{d+1}$, provided that $u$ vanishes outside $Q_{R_0}$.
\end{lemma}

\begin{theorem}
			\label{thm5.7}
Let $p>2$, $a^i=b^i=c=0$. Let $\mu_2$ and $\gamma_2$ be the constants in Corollary \ref{cor5.5}. Then there exists a constant $\gamma_3\in (0,\gamma_2]$ depending only on $d,p$ and $\delta$ such that under Assumption \ref{assump7b} ($\gamma_3$) with $\tilde \pi_{ij}$ in place of $\pi_{ij}$, for any $u\in C_0^\infty(Q_{\mu_2^{-1}R_0})$ and $g\in L_p$ satisfying $\cP u=\Div g$, we have
\begin{equation*}
\|Du\|_{L_p}\le N\|g\|_{L_p}.
\end{equation*}
\end{theorem}
\begin{proof}
Fix $\nu\ge 4$ and $\gamma_3\in (0,\gamma_2]$ to be specified later. Let $\alpha$ be the constant in Lemma \ref{lem5.6} and choose $\tau>1$ such that $p>2\tau$. Lemma \ref{lem5.6} implies
$$
\sum_{i=3}^d (D_i u)^{\#}\le N\nu^{-\alpha}\left(\bM(|Du|^2)\right)^{1/2}
+N\nu^{(d+2)/2}\left(\bM(|g|^2)\right)^{1/2}
$$
$$
+N\nu^{(d+2)/2}\gamma_3^{1/(2\sigma)}\left(\bM(|Du|^{2\tau})\right)^{1/(2\tau)}.
$$
By using the Fefferman--Stein theorem on
sharp functions, and the Hardy--Littlewood maximal function theorem, we get
\begin{equation}
					\label{eq4.24}
\sum_{i=3}^d\|D_i u\|_{L_p}\le N(\nu^{-\alpha}+\nu^{(d+2)/2}\gamma_3^{1/(2\sigma)})\|D u\|_{L_p}
+N\nu^{(d+2)/2}\|g\|_{L_p}.
\end{equation}
Since $\gamma_3\in (0,\gamma_2]$ and $u\in C_0^\infty(Q_{\mu_2^{-1}R_0})$, estimate \eqref{eq4.24} and Corollary \ref{cor5.5} give
\begin{equation}
					\label{eq4.26}
\|D u\|_{L_p}\le N(\nu^{-\alpha}+\nu^{(d+2)/2}\gamma_3^{1/(2\sigma)})\|D u\|_{L_p}
+N\nu^{(d+2)/2}\|g\|_{L_p}.
\end{equation}
We take $\nu$ sufficiently large, and $\gamma_3$ small such that in \eqref{eq4.26}
$$
N(\nu^{-\alpha}+\nu^{(d+2)/2}\gamma_3^{1/(2\sigma)})\le 1/2.
$$
The theorem is proved.
\end{proof}

We are now in the position to prove Theorem \ref{thm5.1}.

\begin{proof}[Proof of Theorem \ref{thm5.1}]
Recall that for simplicity we replace $\pi_{ij}$ in the assumption by $\tilde \pi_{ij}$. The general case can be done by an induction.

The case $p=2$ is classical. Suppose that Assumption \ref{assump7b} is true. For $p>2$, the result follows from Theorem \ref{thm5.7} by using a partition of unity and an idea by S. Agmon; see, for instance, \cite{Krylov_2005}. The case $1<p<2$ follows from a standard duality argument.

For the remaining case that Assumption \ref{assump7} is true, we follow the argument in this subsection and use the results in \cite{DongKim08b} instead of \cite{Dong08}. The theorem is proved.
\end{proof}

As before, we have the following solvability result for elliptic equations in divergence form:
\begin{equation*}
\cL u-\lambda u=\Div g +f,
\end{equation*}
where all the coefficients are time-independent and
$$
\cL u=D_j(a^{ij}D_iu)+D_i(a^{i}u)+ b^{i}D_iu+cu.
$$

\begin{theorem}
                                            \label{thm5.9}
Let $p\in (1,\infty)$, there exists a constant $\gamma=\gamma(d,p,\delta)>0$
such that under Assumption  \ref{assump7b} ($\gamma$) the following holds.

i) For any $ u \in W_p^{1}(\bR^{d})$,
\begin{equation*}
\lambda \|  u \|_{L_p(\bR^d)}
+ \sqrt{\lambda} \|Du\|_{L_p(\bR^{d})}
\le N(\sqrt \lambda+1) \|\cL  u - \lambda  u\|_{\bH^{-1}_p(\bR^{d})}
\end{equation*}
provided that $\lambda\ge \lambda_0$, where $\lambda_0\ge 0$ depends only
on $d$, $p$, $\delta$, $K$ and $R_0$, and $N$ depends only on $d$, $p$ and $\delta$.

ii) For any $\lambda > \lambda_0$ and
$ f, g\in L_p(\bR^{d})$,
there exists a unique $ u \in W_{p}^{1}(\bR^{d})$ solving
$$\cL u - \lambda  u =  f+\Div g$$ in $\bR^{d}$. And $u$ satisfies the estimate
$$
\lambda \|  u \|_{L_p(\bR^{d})}
+ \sqrt{\lambda} \|Du\|_{L_p(\bR^{d})}
\le N\sqrt\lambda\|g\|_{L_p(\bR^{d})} +N\|f\|_{L_p(\bR^{d})}.
$$

iii) In the case that
$$
a^{11}=a^{11}(x^1),\quad
a^{ij}=a^{ij}(x^1,\ldots,x^{\pi_{ij}}),\,\,
ij>1,\quad
a^j=b^j\equiv c\equiv 0,
$$
we can take
$\lambda_0=0$ in i) and ii).
\end{theorem}

\subsection{Solvability of nondivergence form equations}    \label{subsec6.2}
We prove in this subsection the following solvability theorems, which generalize Theorems \ref{thm2} and \ref{thm3} respectively.

\begin{theorem}
                                         \label{thm5.10}

For any $p\in (1,\infty)$, there exists a $\gamma =\gamma(d,\delta,p) >0$
such that under Assumption
\ref{assump7} ($\gamma$) for any $T\in (-\infty,+\infty]$ the following holds.

i) For any $u\in W^{1,2}_p(\bR^{d+1}_T)$,
$$
\lambda\|u\|_{L_{p}(\bR^{d+1}_T)}+\sqrt{\lambda}
\|Du\|_{L_{p}(\bR^{d+1}_T)}+\|D^{2}u\|_{L_{p}(\bR^{d+1}_T)}
+\|u_t\|_{L_{p}(\bR^{d+1}_T)}
$$
\begin{equation}
                                                \label{13.10.21b}
\leq N\|Pu-\lambda u\|_{L_{p}(\bR^{d+1}_T)},
\end{equation}
provided that $\lambda\geq \lambda_0$, where $\lambda_0
\geq0$ depends only
 on $d,\delta,p,K$ and $R_0$, and $N$ depends only on $d,\delta$ and $p$.

ii) For any $\lambda> \lambda_0$ and $f\in L_p(\bR^{d+1}_T)$, there exists a unique solution
$u\in W^{1,2}_p(\bR^{d+1}_T)$ of equation \eqref{parabolic} in
$\bR^{d+1}_T$.

iii) In the case that
$a^{11}=a^{11}(t)$, and
$$
a^{ij}=a^{ij}(t,x^1,\ldots,x^{\pi_{ij}}),\,\,
ij>1,\quad
b^j\equiv c\equiv 0,
$$  we can take
$\lambda_0=0$ in i) and ii).
\end{theorem}

\begin{theorem}
                                         \label{thm5.10b}

For any $p\in (1,\infty)$, there exists a $\gamma =\gamma(d,\delta,p) >0$
such that under Assumption \ref{assump7b} ($\gamma$) for any $T\in (-\infty,+\infty]$ the following holds.

i) For any $u\in W^{1,2}_p(\bR^{d+1}_T)$, we have \eqref{13.10.21b}
provided that $\lambda\geq \lambda_0$, where $\lambda_0
\geq0$ depends only
 on $d,\delta,p,K$ and $R_0$, and $N$ depends only on $d,\delta$ and $p$.

ii) For any $\lambda> \lambda_0$ and $f\in L_p(\bR^{d+1}_T)$, there exists a unique solution
$u\in W^{1,2}_p(\bR^{d+1}_T)$ of equation \eqref{parabolic} in
$\bR^{d+1}_T$.

iii) In the case that
$a^{11}=a^{11}(x^1)$, and
$$
a^{ij}=a^{ij}(t,x^1,\ldots,x^{\pi_{ij}}),\,\,
ij>1,\quad
b^j\equiv c\equiv 0,
$$  we can take
$\lambda_0=0$ in i) and ii).
\end{theorem}

As a consequence of Theorem \ref{thm5.10b}, we have the following solvability theorem for elliptic equations.

\begin{theorem}
                                         \label{thm5.10c}

For any $p\in (1,\infty)$, there exists a $\gamma =\gamma(d,\delta,p) >0$ such that under Assumption \ref{assump7b} ($\gamma$) the following holds.

i) For any $u\in W^{2}_p(\bR^{d})$, we have
\begin{equation*}
\lambda\|u\|_{L_{p}(\bR^{d})}+\sqrt{\lambda}
\|Du\|_{L_{p}(\bR^{d})}+\|D^{2}u\|_{L_{p}(\bR^{d})}
\leq N\|Lu-\lambda u\|_{L_{p}(\bR^{d})},
\end{equation*}
provided that $\lambda\geq \lambda_0$, where $\lambda_0
\geq0$ depends only
 on $d,\delta,p,K$ and $R_0$, and $N$ depends only on $d,\delta$ and $p$.

ii) For any $\lambda> \lambda_0$ and $f\in L_p(\bR^{d})$, there exists a unique solution
$u\in W^{2}_p(\bR^{d})$ of equation \eqref{elliptic} in
$\bR^{d}$.

iii) In the case that
$a^{11}=a^{11}(x^1)$, and
$$
a^{ij}=a^{ij}(x^1,\ldots,x^{\pi_{ij}}),\,\,
ij>1,\quad
b^j\equiv c\equiv 0,
$$  we can take
$\lambda_0=0$ in i) and ii).
\end{theorem}

We will only prove Theorem \ref{thm5.10b}. The proof of Theorem \ref{thm5.10} is similar and actually simpler. We make a few preparations before the proof.

Denote
$$
P_0u=-u_t+a^{ij}D_{ij}u,
$$
where $a^{11}=a^{11}(x^1)$ and $a^{ij}=a^{ij}(t,x^1,\ldots,x^{\pi_{ij}})$ for $ij>1$.
We first establish a solvability result for equations with these simple leading coefficients.
\begin{theorem}
						\label{thm5.11}
Let $p\in (1,\infty)$ and $T\in (-\infty,\infty]$. Then for any $u\in W^{1,2}_p(\bR^{d+1}_T)$ and $\lambda\ge 0$, we have
$$
\lambda\|u\|_{L_{p}(\bR^{d+1}_T)}+\sqrt{\lambda}
\|Du\|_{L_{p}(\bR^{d+1}_T)}+\|D^2u\|_{L_{p}(\bR^{d+1}_T)}
+\|u_{t}\|_{L_{p}(\bR^{d+1}_T)}
$$
\begin{equation}
                \label{eq10.30c}
\le N\|P_0u-\lambda u\|_{L_{p}(\bR^{d+1}_T)},
\end{equation}
where $N=N(d,p,\delta)>0$.
Moreover, for any $f\in L_p(\bR^{d+1}_T)$ and $\lambda>0$ there is a unique $u\in W^{1,2}_p(\bR^{d+1}_T)$ solving
$$
P_0 u-\lambda  u=f \quad\text{in}\,\,\, \bR^{d+1}_T.
$$
\end{theorem}
\begin{proof}
We follow closely the proof of Theorem \ref{thm3.2}, and recall that we only need to prove \eqref{eq10.30c} for $u\in C_0^\infty$ and $T=\infty$. We make a change of variables as in the proof of Theorem \ref{thm3.2}. Denote
\begin{align*}
v(t,y^1,y')&=u(t,\phi^{-1}(y^1),y'),\\
\bar a^{11}(y^1)&=a^{11}(\phi^{-1}(y^1)),\\
\bar a^{ij}(t,y)&=a^{ij}(t,\phi^{-1}(y^1),y^2,\ldots,y^{\pi_{ij}}),\,\,ij>1,\\
\bar f(t,y)&=f(t,\phi^{-1}(y^1),y').
\end{align*}
Define a divergence form operator $\bar P_0$ in the $(t,y)$-coordinates by
$$
\bar P_0 v=-v_t+D_{1}\left(\frac 1 {\bar a^{11}} D_1\bar v\right)+\sum_{j=2}^d D_{j}\left(\frac {\bar a^{1j}+\bar a^{j1}}
{\bar a^{11}} D_1v
\right)
$$
$$
+\sum_{1<i<j} D_{j}\left((\bar a^{ij}+\bar a^{ji})D_iv\right)
+\sum_{i=2}^d D_{i}(\bar a^{ii}D_iv).
$$
Clearly, $v$ satisfies in $\bR^{d+1}$
\begin{equation}
                            \label{eq6.52}
\bar P_0 v-\lambda v=\bar f.
\end{equation}
This is a divergence form equation, which satisfies the condition of Theorem \ref{thm5.1} iii). Thus, we have
\begin{equation*}
\lambda \| v \|_{L_p}
+ \sqrt\lambda\|Dv\|_{L_p}
\le N\|\bar f\|_{L_p}.
\end{equation*}
Therefore,
\begin{equation}
					\label{eq5.41}
\lambda \|u \|_{L_p}
+ \sqrt\lambda\|Du\|_{L_p}
\le N\|f\|_{L_p}.
\end{equation}
In order to estimate the second derivatives, notice that $D_dv$ satisfies
$$
\bar P_0 (D_d v)-\lambda D_d v=D_d \bar f.
$$
Again by Theorem \ref{thm5.1} iii), we get
$$
\| D_{yy^d} v \|_{L_p}
\le N\|\bar f\|_{L_p},
$$
which implies
\begin{equation}
				\label{eq5.42}
\| D_{xx^d} u \|_{L_p}
\le N\|f\|_{L_p}.
\end{equation}
Next, we estimate $D_{xx^{d-1}} u$. Upon moving the terms involving $D_d$ to the right-hand side of \eqref{eq6.52}, adding $D_{d}^2v$ to both sides of \eqref{eq6.52}, and taking one derivative in $y_{d-1}$ on both sides, we get
$$
\bar P_1 (D_{d-1} v)-\lambda D_{d-1} v=D_{d-1} \bar f-
\sum_{i=1}^{d-1} D_{d-1}\left((\bar a^{id}+\bar a^{di})D_{id}v\right)
$$
$$
+D_{d-1}\left(\frac {\bar a^{1d}+\bar a^{d1}}
{\bar a^{11}} D_{1d}v
\right)
+D_{d-1}((1-\bar a^{dd})D_{d}^2v),
$$
where $\bar P_1$ is defined by
$$
\bar P_1 v=-v_t+D_{1}\left(\frac 1 {\bar a^{11}} D_1\bar v\right)+\sum_{j=2}^{d-1} D_{j}\left(\frac {\bar a^{1j}+\bar a^{j1}}
{\bar a^{11}} D_1v
\right)
$$
$$
+\sum_{1<i<j<d} D_{j}\left((\bar a^{ij}+\bar a^{ji})D_iv\right)
+\sum_{i=2}^{d-1} D_{i}(\bar a^{ii}D_iv)+D_{d}^2 v.
$$
By Theorem \ref{thm5.1} iii), we get
$$
\| D_{yy^{d-1}} v \|_{L_p}
\le N\|\bar f\|_{L_p}+N\| D_{yy^{d}} v \|_{L_p},
$$
which together with \eqref{eq5.42} implies
\begin{equation*}
\| D_{xx^{d-1}} u \|_{L_p}
\le N\|f\|_{L_p}.
\end{equation*}
We can repeat this procedure $d-1$ times, and then obtain
\begin{equation}
				\label{eq7.12}
\| D_{xx'} u \|_{L_p}
\le N\|f\|_{L_p}.
\end{equation}

Finally, from \eqref{eq2.32}, we see that $w:=D_1 u$ satisfies
\begin{equation*}
-w_t+D_1(a^{11}D_1w)+\Delta_{d-1}w-\lambda w=D_1 f+\sum_{ij>1}D_1
\left((\delta_{ij}-a^{ij})D_{ij}u\right).
\end{equation*}
We use Theorem \ref{thm5.1} iii) one more time to get
\begin{equation}
				\label{eq5.44}
\| D_{1}^2 u \|_{L_p}\le \| Dw \|_{L_p}\le N\| f \|_{L_p}
+N\sum_{ij>1} \| D_{ij}u \|_{L_p}.
\end{equation}
Combining \eqref{eq5.41}, \eqref{eq7.12}, \eqref{eq5.44} and \eqref{eq3.03} yields \eqref{eq10.30c}.
The theorem is proved.
\end{proof}

In the sequel, we again only consider the case $d=3$, or $d>3$ and $\pi_{ij}$ is replaced by $\tilde \pi_{ij}$ in Assumption \ref{assump7b}. Like before, the general case follows by an induction.

\begin{theorem}
                                            \label{thm5.12}
Let $d\geq 2$, $p\in (1,\infty)$, $\gamma > 0$, $\tau,\sigma \in (1,\infty)$ satisfying
$1/\tau+1/\sigma=1$. Assume $b^i=c=0$ and $u\in W^{1,2}_p$. Then under
Assumption \ref{assump7b} ($\gamma$) with $\tilde \pi_{ij}$ in place of $\pi_{ij}$,  there exist positive constants $\alpha=\alpha(d,\delta)\in (0,1)$ and $N$ depending only on
$d$, $p$, $\delta$, and $\tau$ such that, for any
$(t_0,x_0)\in \bR^{d+1}$,
$r\in (0,\infty)$, and $\kappa\ge 4$,
\begin{multline*}
\sum_{i=3}^d\left(|D_{i}^2 u(t,x)-(D_{i}^2u)_{Q_r(t_0,x_0)}|^p
\right)_{Q_r(t_0,x_0)}
\le N\kappa^{d+2}
\left(|P u|^p\right)_{Q_{\kappa r}(t_0,x_0)}\\
+N\kappa^{d+2}\gamma^{1/\sigma}
\left(|D^2 u|^{p\tau}\right)_{Q_{\kappa r}(t_0,x_0)}^{1/\tau}
+N\kappa^{-p\alpha}\sum_{i=3}^d
\left(|D_{i}^2 u|^p\right)_{Q_{\kappa r}(t_0,x_0)},
\end{multline*}
provided that $u$ vanishes outside  $Q_{R_0}$.
\end{theorem}
\begin{proof}
The proof is similar to that of Theorem \ref{thm4.1} by using Theorem \ref{thm5.11} instead of Theorem \ref{thm2.2}. We omit the details.
\end{proof}

\begin{lemma}
 						\label{lem5.13}
Assume $b^i=c=0$. For any
$p\in (1,\infty)$ there exists a $\gamma_2$, $\mu_2$ and $N$, depending only on $d$, $\delta$ and $p$, such that under Assumption \ref{assump7b} ($\gamma_2$) with $\tilde \pi_{ij}$ in place of $\pi_{ij}$, for any $u\in C_0^\infty(Q_{\mu_2^{-1}R_0})$ we have
\begin{equation}
                \label{eq7.34}
\|D^2 u\|_{L_{p}}+\|u_{t}\|_{L_{p}}\le N\|Pu\|_{L_{p}}+N\sum_{i=3}^d\|D_{i}^2 u\|_{L_{p}}.
\end{equation}
\end{lemma}
\begin{proof}
Set $f=Pu$. Note that $u$ satisfies
\begin{equation*}
-u_t+\sum_{i,j=1}^2 a^{ij}D_{ij}u+\sum_{i=3}^dD^2_{i}u
=f+\sum_{\max(i,j)>2} (\delta_{ij}-a^{ij})D_{ij}u.
\end{equation*}
The coefficients on the left-hand side above satisfy Assumption \ref{assump2} ($\gamma_2$) ii) and iii).
Thus by Theorem \ref{thm5.2} and Lemma \ref{lem5.2}, for $\gamma_2$ sufficiently small and $\mu_2$ sufficiently large depending only on $d$, $\delta$ and $p$, we have
$$
\|u_t\|_{L_p}+\|D^2 u\|_{L_p}\le N\|f\|_{L_p}+N\sum_{\max(i,j)>2}\|D_{ij}u\|_{L_p}.
$$
By an inequality similar to \eqref{3.57} we reach \eqref{eq7.34}. The lemma is proved.
\end{proof}

\begin{proof}[Proof of Theorem \ref{thm5.10}]
Assume for the moment that in Assumption \ref{assump7b} $\pi_{ij}$ is replace by $\tilde \pi_{ij}$. We first establish a counterpart of Lemma \ref{lem4.2} by using Theorem \ref{thm5.12} and Lemma \ref{lem5.13}. After that we prove the theorem by arguing as in the proof of Theorem \ref{thm2}. The general case follows by an induction. The theorem is proved.
\end{proof}

\appendix

\section{Estimates for divergence form equations}

For reader's convenience, in the Appendix section, we recall a few results obtained in \cite{Dong08} and \cite{DongKim08b} for divergence form parabolic equations and systems. These results have been used in this paper.

First, consider parabolic equations in divergence form:
\begin{equation}
                                                \label{divpara}
-u_t+D_j(a^{ij}D_iu)-\lambda u=\Div g +f,
\end{equation}
where $\lambda\ge 0$ is a constant and $g=(g^1,g^2,\ldots,g^d)$.
We assume the same boundedness and ellipticity condition on $a^{ij}$ as in the introduction.

\begin{proposition}[Corollary 5.5 of \cite{Dong08}]
                                                \label{lemA1}
Let $p\in (1,\infty)$, $T \in (-\infty, \infty]$,  and $f,g\in  L_p(\bR^{d+1}_T)$. Assume that $a^{11}=a^{11}(x^1)$ and $a^{ij}=a^{ij}(t,x^1)$ for $ij>1$. Then there exists a constant $N>0$, depending only
on $d$, $p$ and $\delta$,
such that for any $ u \in \cH_p^{1}(\bR^{d+1}_T)$ satisfying \eqref{divpara} we have
\begin{equation*}
\lambda \|  u \|_{L_p(\bR^{d+1}_T)}
+ \sqrt\lambda\|Du \|_{L_p(\bR^{d+1}_T)}
\le N\sqrt\lambda\|g\|_{L_p(\bR^{d+1}_T)}+N\|f\|_{L_p(\bR^{d+1}_T)},
\end{equation*}
provided that $\lambda \ge 0$. In particular, when $\lambda=0$ and $f\equiv 0$, we have
\begin{equation*}
\|Du \|_{L_p(\bR^{d+1}_T)}
\le N\|g\|_{L_p(\bR^{d+1}_T)}.
\end{equation*}
\end{proposition}

The following result follows directly from the proof of Theorem 2.5 \cite{Dong08}.

\begin{proposition}
                                            \label{lemA2}
Let $p\in (1,\infty)$, $\lambda=0$, $f\equiv 0$ and $g\in L_p$.
Then there exist constants $\gamma_1$, $N$ and $\mu_1>0$ depending only
on $d$, $p$, and $\delta$,
such that under Assumption \ref{assump2} ($\gamma_1$) ii) and iii) the following holds.
For any $ u \in C_0^\infty$ vanishing outside $Q_{\mu^{-1}_1 R_0}$ and satisfying \eqref{divpara} in $\bR^{d+1}$, we have
$$
\|Du \|_{L_p}\le N\|g \|_{L_p}.
$$
\end{proposition}

Next, we consider parabolic systems in the form
\begin{equation}
                        \label{eq15.39}
-\vu_t + D_\alpha(A^{\alpha \beta}D_{\beta}\vu)-\lambda \vu=\Div \vg,
\end{equation}
where $\vu = (u^1, \ldots, u^m)^T$,
$$
\vg_{\alpha} = (g_{\alpha}^1, \ldots, g_{\alpha}^m)^{\text{tr}},
\quad \alpha = 1, \ldots, d,
$$
and $A^{\alpha \beta}=[A^{\alpha\beta}_{ij}]_{i,j=1}^d$ satisfies \eqref{eq15.41}. Moreover, we suppose
$$
A^{11}=A^{11}(t),\quad
A^{\alpha\beta}=A^{\alpha\beta}(t,x^1),\,\,\alpha\beta>1.
$$

The proposition below is a special case of Theorem 5.1 \cite{DongKim08b}.

\begin{proposition}
                                \label{lemA3}
Let $p\in (1,\infty)$, $\lambda\in (0,\infty)$, $T\in (-\infty,\infty]$ and $\vg \in L_p$. Suppose that
$\vu \in \cH_p^1(\bR^{d+1}_{T})$ satisfies \eqref{eq15.39} in $\bR^{d+1}_T$.
Then,
\begin{equation*}
\|D \vu\|_{L_p(\bR^{d+1}_T)}
+\sqrt{\lambda} \|\vu\|_{L_p(\bR^{d+1}_T)}
\leq N\|\vg\|_{L_p(\bR^{d+1}_T)},
\end{equation*}
where $N=N(d,m,\delta,p)$.
\end{proposition}

The following H\"older estimate is proved in \cite{DongKim08b} by using a bootstrap argument.

\begin{proposition}[Lemma 6.3 \cite{DongKim08b}]
                                \label{lemA4}
Let $p\in (1,\infty)$. Assume $\vu\in C_{\text{loc}}^{\infty}$
satisfies \eqref{eq15.39} in $Q_{2}$ with $\lambda=0$ and $\vg\equiv 0$.
Then for any $\gamma\in (0,1)$, we have
$$
\|\vu\|_{C^{\gamma/2,\gamma}(Q_1)}
+ \|D_{x'}\vu\|_{C^{\gamma/2,\gamma}(Q_1)}\le N\|\vu\|_{L_p(Q_{2})},
$$
where $N=N(d,m,\delta,\gamma,p)$.
\end{proposition}

\section*{Acknowledgement}

The author is grateful to Nicolai V. Krylov and Doyoon Kim for very helpful
discussions and comments during the preparation of this paper.
He also would like to thank the referee and the associate editor
for their helpful comments on an earlier version of the paper.



\begin{thebibliography}{m}

\bibitem{AM07} E. Acerbi, G. Mingione, Gradient estimates for a class of parabolic systems, \textit{Duke Math. J.} \textbf{136} (2007), no. 2, 285--320.

\bibitem{BC93} M. Bramanti, M. Cerutti, $W_p^{1,2}$
solvability for the Cauchy-Dirichlet problem for parabolic
equations with VMO coefficients, \textit{Comm. Partial Differential
Equations} \textbf{18} (1993), no. 9-10, 1735--1763.

\bibitem{ByunWang08} S. Byun, L. Wang, Gradient estimates for elliptic systems in non-smooth domains,  \textit{Math. Ann.}  \textbf{341}  (2008),  no. 3, 629--650.

\bibitem{CFL1} F. Chiarenza, M. Frasca, P. Longo,
Interior $W^{2,p}$ estimates for nondivergence elliptic
equations with discontinuous coefficients, \textit{Ricerche Mat.} \textbf{40} (1991), no. 1, 149--168.

\bibitem{CFL2} \bysame,
$W^{2,p}$-solvability of the Dirichlet problem for
nondivergence elliptic equations with VMO coefficients,
\textit{Trans. Amer. Math. Soc.} \textbf{336} (1993), no. 2, 841--853.

\bibitem{Chi} G. Chiti, A $W^{2,2}$ bound for a class of
elliptic equations in nondivergence form with rough
coefficients, \textit{Invent. Math.} \textbf{33} (1976), no. 1, 55--60.

\bibitem{Dong08} H. Dong, Solvability of parabolic equations in divergence form with partially BMO coefficients, \textit{J. Funct. Anal.} \textbf{258} (2010),  no. 7,  2145--2172..

\bibitem{DK08} H. Dong, D. Kim, Parabolic and
elliptic systems with VMO coefficients, \textit{Methods Appl. Anal.} \textbf{16} (2009), no. 3, 365--388.

\bibitem{DongKim08a} \bysame, Elliptic equations in divergence form with partially BMO coefficients, \textit{Arch. Ration. Mech. Anal.} \textbf{196}  (2010),  no. 1, 25--70.

\bibitem{DongKim08b} \bysame, Parabolic and elliptic systems in divergence form with partially BMO coefficients, \textit{Calc. Var. Partial Differential Equations}, to appear (2010), DOI: 10.1007/s00526-010-0344-0.


\bibitem{DongKrylov} H. Dong, N. V. Krylov, Second-order elliptic and parabolic equations with $B(\mathbb R^{2}, VMO)$ coefficients, \textit{Trans. Amer. Math. Soc.} \textbf{362}  (2010), no. 12, 6477--6494.

\bibitem{GT98} D. Gilbarg, N. S. Trudinger, ``Elliptic partial differential equations of second order'', Reprint of the 1998 edition, Classics in Mathematics, Springer-Verlag, Berlin, 2001.

\bibitem{HHH} R. Haller-Dintelmann, H. Heck, M. Hieber, $L^p$--$L^q$-estimates for parabolic systems in non-divergence form with VMO coefficients, \textit{J. London Math. Soc. (2)} \textbf{74} (2006), no. 3, 717--736.

\bibitem{Kim07} D. Kim, Parabolic equations with measurable coefficients. II. (English summary) \textit{J. Math. Anal. Appl.} \textbf{334} (2007), no. 1, 534--548.

\bibitem{Kim07a} \bysame, Elliptic and parabolic equations with measurable coefficients in $L_p$-spaces with mixed norms, \textit{Methods Appl. Anal.} \textbf{15}  (2008),  no. 4, 437--467.

\bibitem{Kim07b} \bysame, Parabolic equations with partially BMO coefficients and boundary value problems in Sobolev spaces with mixed norms, \textit{Potential Anal.} \textbf{33} (2010), no. 1, 17--46.

\bibitem{KimKrylov07} D. Kim, N. V. Krylov, Elliptic differential equations with coefficients measurable with respect to one variable and VMO with respect to the others, \textit{SIAM J. Math. Anal.} \textbf{39} (2007), no. 2, 489--506.

\bibitem{KK2} \bysame, Parabolic equations with measurable coefficients, \textit{Potential Anal.} \textbf{26} (2007), no. 4, 345--361.

\bibitem{kunst} P.C. Kunstmann, On maximal regularity of type $L^p$-$L^q$ under minimal assumptions for elliptic non-divergence operators, \textit{J. Funct. Anal.} \textbf{255}  (2008),  no. 10, 2732--2759.

\bibitem{Kr85}  N.~V. Krylov,
  ``Nonlinear elliptic and parabolic equations of second
  order'',  Nauka, Moscow, 1985 in Russian; English translation
  Reidel, Dordrecht, 1987.

\bibitem{Krylov_2005}
\bysame, Parabolic and elliptic equations with {VMO} coefficients, \textit{Comm. Partial Differential Equations} \textbf{32} (2007), no. 3, 453--475.

\bibitem{Krylov_2007_mixed_VMO}
\bysame, Parabolic equations with {VMO} coefficients in Sobolev spaces with mixed  norms, \textit{J. Funct. Anal.} \textbf{250} (2007), no. 2,  521--558.

\bibitem{Krylov08} \bysame, Second-order elliptic
equations with variably partially VMO coefficients, \textit{J. Funct. Anal.} \textbf{257}  (2009),  no. 6, 1695--1712.

\bibitem{Li} G. Lieberman, ``Second order parabolic differential
equations'', World Scientific, Singapore-New Jersey-London-Hong
Kong, 1996.

\bibitem{Lo72a} A. Lorenzi, On elliptic equations with
piecewise constant coefficients, \textit{Applicable Anal.} \textbf{2} (1972), 79--96.

\bibitem{MaPaSo00} A. Maugeri, D. Palagachev, L. Softova,
``Elliptic and parabolic equations with discontinuous coefficients'', Math.
Res., vol. 109, Wiley---VCH, Berlin, 2000.


\bibitem{PS2} D. Palagachev, L. Softova, Characterization of the interior regularity for parabolic systems with discontinuous coefficients, \textit{Atti. Accad. Naz. Lincei Cl. Sci. Fis. Mat. Natur. Rend. Lincei (9) Mat. Appl.} \textbf{16} (2005), 125--132.

\end{thebibliography}
\end{document}